\documentclass[11pt,twoside,reqno]{amsart}
\usepackage{amsmath,amsfonts,amsthm,amssymb}
\usepackage{graphicx,psfrag,overpic}
\usepackage{bm}
\usepackage[all,cmtip,line]{xy}
\usepackage{array,setspace}
\usepackage{color}
\usepackage{xcolor}
\numberwithin{equation}{section}
\numberwithin{figure}{section}

\usepackage{epsfig}
\newtheorem{theorem}{Theorem}[section]

\newtheorem{lemma}{Lemma}[section]
\newtheorem{proposition}[theorem]{Proposition}
\newtheorem{remark}{Remark}[section]
\newtheorem{definition}{Definition}[section]

\renewcommand{\u}{{ u}}

\newcommand{\supp}{\mbox{\rm supp}}

\usepackage{easyReview}

\newcommand \R{\mathbb{R}}
\newcommand{\dd}{{\rm d}}

\begin{document}

\title[Isothermal Limit for Isentropic Gas Dynamics]
{Isothermal Limit of Entropy Solutions of the Euler Equations for Isentropic Gas Dynamics}

\author{Gui-Qiang G. Chen}
\address{G.-Q. Chen,
Mathematical Institute, University of Oxford\\
Oxford, OX2 6GG, UK;
School of Mathematical Sciences, Fudan University\\
Shanghai 200433, China;\\
Academy of Mathematics and Systems Science\\
Chinese Academy of Sciences, Beijing 100190, China}
\email{chengq@maths.ox.ac.uk}

\author{Fei-Min Huang}
\address{F.-M. Huang, Institute of Applied Mathematics\\
 Academy of Mathematics and Systems Science\\
Chinese Academy of Sciences, Beijing 100190, China}
\email{fhuang@amt.ac.cn}

\author{Tian-Yi Wang}
\address{T.-Y. Wang, Department of Mathematics, School of Science, Wuhan University of Technology,
	Wuhan, Hubei 430070, China}
\email{tianyiwang@whut.edu.cn; wangtianyi@amss.ac.cn}
\date{\today}

\begin{abstract}
We are concerned with the isothermal limit of entropy solutions in $L^\infty$, containing the vacuum states,
of the Euler equations for isentropic gas dynamics.
We prove that the entropy solutions in $L^\infty$ of the isentropic Euler equations
converge strongly to the corresponding entropy solutions of the isothermal Euler equations,
when the adiabatic exponent $\gamma \rightarrow 1$.
This is achieved by combining careful entropy analysis and refined kinetic formulation
with compensated compactness argument to obtain the required uniform estimates
for the limit. The entropy analysis involves careful estimates
for the relation between the corresponding entropy pairs for the isentropic and isothermal Euler equations
when the adiabatic exponent $\gamma\to 1$.
The kinetic formulation for the entropy solutions of the isentropic Euler equations
with the uniformly bounded initial data
is refined, so that the total variation of the dissipation measures in the formulation
is locally uniformly bounded with respect to $\gamma>1$.
The explicit asymptotic analysis of the Riemann solutions containing the vacuum states
is also presented.
\end{abstract}

\maketitle

\bigskip
\section{Introduction}
We are concerned with the isothermal limit of entropy solutions in $L^\infty$, containing the vacuum states,
of the Euler equations for isentropic gas dynamics which is
the oldest, but still most prominent,
paradigm for the analysis of hyperbolic systems of conservation laws.
The one-dimensional Euler equations for barotropic gas dynamics take the form:
\begin{eqnarray}\label{1.5}
\begin{cases}
\partial_t\rho+\partial_x m =0,\\[1mm]
\partial_t m +\partial_x\big(\frac{m^2}{\rho}+p\big)=0,\\
\end{cases}
\end{eqnarray}
where $\rho$ denotes the fluid density, $p=p(\rho)$ is the pressure, and $m=\rho u$ is the momentum.
When $\rho>0$, $u:=\frac{m}{\rho}$ represents the fluid velocity.

\smallskip
The pressure-density relation under consideration can be written by scaling as
\begin{equation}\label{1.2a}
	p=\frac{\rho^{\gamma}}{\gamma}  \qquad \mbox{for $\gamma> 1$}
\end{equation}
for ideal isentropic gases,
while
\begin{equation}\label{1.2b}
	p=\rho
\end{equation}
for the isothermal gas.

Consider the Cauchy problem for system \eqref{1.5}
with large Cauchy data:
\begin{equation}\label{1.5b}
(\rho, m)|_{t=0}=(\rho_0, m_0)(x),
\end{equation}
with $\rho_0(x)\ge 0$.

\smallskip
The global existence of solutions with large initial data
in $L^{\infty}$ was first established in DiPerna \cite{DiPerna}
for $\gamma=1+\frac{2}{2n+1}$ with $n\geq 2$ integer.
For the general interval $1<\gamma \leq \frac{5}{3}$,
the global existence problem was solved in
Ding-Chen-Luo \cite{DCL,DCL2}
and Chen \cite{Chen}.
The adiabatic exponent range $\gamma\ge 3$ was solved in Lions-Perthame-Tadmor \cite{LPT},
and the remaining case $\frac{5}{3}<\gamma<3$ was closed in Lions-Perthame-Souganidis \cite{LPS}.
Later, the existence problem for the general pressure $p(\rho)$ with
\eqref{1.2a} as its leading asymptotic term
as $\rho\to 0$
was solved in Chen-LeFloch \cite{CLone},
whose approach further simplifies the proofs for the $\gamma$-law case for all $\gamma>1$.

\smallskip
For the isothermal case \eqref{1.2b},
the first existence result was obtained in Nishida \cite{Nishida}
for $BV$ solutions without vacuum for large initial data.
For general $L^\infty$ solutions containing the vacuum states,
it was first solved in Huang-Wang \cite{HuangWang}; see also \cite{LS} for a different approach.

\smallskip
For the isothermal limit when the adiabatic exponent $\gamma\rightarrow 1$,
the previous analysis is in the framework of $BV$ solutions away from the vacuum.
The first result was established by Nishida-Smoller \cite{Nishida-Smoller} in the Lagrangian coordinates.
Chen-Christoforou-Zhang  \cite{ChenCZhang1} established the
$L^1$--dependence estimate of the $BV$ solutions with respect to $\gamma>1$ in the Eulerian coordinates.
The isothermal limit for the BV solutions from the full Euler equations away from the vacuum
was solved by Chen-Christoforou-Zhang in \cite{ChenCZhang2}.
It is well-known that the vacuum states occur generically in the entropy solutions
of the Euler equations, even starting with the initial data without vacuum states.
A prototypical example is the Riemann problem with Riemann initial data away from the vacuum for which
there exists a global Riemann solution consisting of two rarefaction waves with an intermediate
vacuum state for the isentropic Euler equations; see \cite{Chen2000} and \S 7 below.
For the Riemann solution containing the vacuum states,
the fluid velocity $u$ is bounded for $\gamma>1$ near the vacuum,
however, $|u|$ goes to $\infty$ as the density goes to $0$ for $\gamma=1$; see also \S 7.
Another difference is the sound speed as the speed of pressure disturbance travelling through the medium
vanishes as the density goes to $0$ for $\gamma>1$, while the sound speed $c$ is a fixed constant
independent of the density for $\gamma=1$.
Therefore, the isothermal limit of entropy solutions in $L^\infty$, containing the vacuum states,
with large initial data has been a longstanding open problem in the analysis of nonlinear
partial differential equations and mathematical fluid mechanics.

\smallskip
The main objective of this paper is to provide an affirmative answer to this problem by providing
a rigorous convergence proof of the isothermal limit.
More precisely, we prove that the entropy solutions in $L^\infty$,  containing the vacuum states,
of the isentropic Euler equations
converge strongly to the corresponding entropy solutions of the isothermal Euler equations
when the adiabatic exponent $\gamma \rightarrow 1$.
This is achieved by combining careful entropy analysis and refined kinetic formulation
with compensated compactness argument to obtain the required uniform estimates
for the isothermal limit.
These effective approaches for the analysis of the isentropic Euler equations
are based on the subtle
analysis for the Euler–Poisson–Darboux
equation, while the entropy equation of the isothermal Euler equations
is not governed by the Euler-Poisson-Darboux equation.
The entropy analysis involves careful estimates
for the relation between the corresponding entropy pairs for the isentropic and isothermal Euler equations
when the adiabatic exponent $\gamma\to 1$.
The kinetic formulation for the entropy solutions of the isentropic Euler equations
with the uniformly bounded initial data
is refined, so that the total variation of the dissipation measures in the formulation
is locally uniformly bounded with respect to $\gamma>1$.
The compensated compactness argument is based on
the $H^{-1}$--compactness of entropy dissipation measures,
the div-curl lemma in \cite{Murat,Tartar}, and the Young measure presentation
theorem ({\it cf}. \cite{Ball,Tartar}); see also \cite{Dafermos,Evans}.

\smallskip
The rest of this paper is organized as follows:
In \S 2, we state some existence results and properties of entropy solutions
of the compressible Euler equations and present the main theorem of this paper, Theorem 2.3,
regarding the isothermal limit of the entropy solutions.
The uniform $L^\infty$ estimate of the entropy solutions with respect
to the adiabatic exponents $\gamma>1$ is shown in \S 3,
while the refined kinetic formulation is obtained in \S 4.
The uniform relation between the corresponding entropy pairs
for the isentropic and isothermal Euler equations
is established in \S 5.
In \S 6, we complete the proof of the main theorem, Theorem 2.3.
Finally, in \S 7, the explicit asymptotic analysis of the Riemann solutions
containing the vacuum states
is presented. In particular, the phenomenon of decavitation is shown as $\gamma\to 1$,
which is different from the formation of cavitation and concentration
in the vanishing pressure limit (equivalently, the high Mach limit) presented
in Chen-Liu \cite{ChenLiu}.

\section{Entropy Solutions of the Compressible Euler Equations}

In this section, we first present some existence results and properties of entropy solutions
of the compressible Euler equations \eqref{1.5}, which can be rewritten as
a hyperbolic system of conservation laws of the form:
\begin{equation}\label{E-1}
\partial_tU +\partial_x F^{(\theta)}(U)=0
\end{equation}
with $U=(\rho, m)$, $F^{(\theta)}(U)=(m, \frac{m^2}{\rho}+p^{(\theta)}(\rho))$,
and $\theta:=\frac{\gamma-1}{2}\ge 0$.
The pressure-density relation \eqref{1.2a} can be rewritten as
\begin{equation} \label{1.2c}
p^{(\theta)}(\rho)=\frac{\rho^{2\theta+1}}{2\theta+1}.
\end{equation}
System \eqref{E-1} represents the isentropic Euler equations when $\theta>0$
and the isothermal Euler equations when $\theta=0$.
Then the entropy pair $(\eta^{(\theta)}_*, q^{(\theta)}_*)$ of the mechanical energy and energy flux is of the form:
\begin{equation}\label{Energy}
(\eta^{(\theta)}_*, q^{(\theta)}_*)(\rho, m)
=(\frac{1}{2}\frac{m^2}{\rho}+ \rho e^{(\theta)}(\rho),	\,
	\frac{1}{2}\frac{m^3}{\rho^2}+ m \frac{{\rm d}(\rho e^{(\theta)}(\rho))}{{\rm d}\rho}),
\end{equation}
where $e^{(\theta)}(\rho)=\int^\rho_1\frac{p^{(\theta)}(\tau)}{\tau^2} d \tau$ is the specific internal energy.
Notice that $c^{(\theta)}(\rho)=\rho^\theta$, which implies that
$\lim_{\rho\rightarrow0}c^{(\theta)}(\rho)=0$, while $c^{(0)}(\rho)\equiv1$.

We now introduce the following broad class of entropy solutions:

\begin{definition}
Let $\theta \geq 0$.
A vector function $U(t,x)=(\rho, m)(t,x)\in L^\infty(\mathbb{R}^2_+)$ is called
an entropy solution of the Cauchy problem \eqref{1.5}--\eqref{1.5b}
with initial data $(\rho_0, m_0)(x)\in L^\infty(\mathbb{R})$ if,
for any test function $\phi\in C_0^\infty(\mathbb{R}^2_+)$,
\begin{eqnarray}
&&\int_{\mathbb{R}_+^2}\big(U\,\partial_t\phi+ F^{(\theta)}(U)\,\partial_x\phi\big)\,\dd x \dd t
+\int_{\mathbb{R}}U_0(x)\phi(0,x)\,\dd x=0,\label{2.8a}
\end{eqnarray}
and, for any test function $\phi\in C_0^\infty(\mathbb{R}^2_+)$ with $\phi\geq0$,
\begin{eqnarray}
\int_{\mathbb{R}_+^2}\big(\eta^{(\theta)}_*(U)\,\partial_t\phi
+q^{(\theta)}_*(U)\,\partial_x\phi\big)\, \dd x \dd t
+ \int_{\mathbb{R}}\eta^{(\theta)}_*(U_0)(x)\phi(0,x)\,\dd x\geq0.
\label{MechanicalEnergy}
\end{eqnarray}
\end{definition}

For system \eqref{E-1}, a general entropy pair $(\eta,q)$ obeys the following linear hyperbolic system:
\begin{equation}\label{EntropyPair}
\nabla q(\rho, m)=\nabla\eta (\rho, m)\nabla F^{(\theta)}(\rho, m).
\end{equation}
A weak entropy is an entropy $\eta(\rho, m)$ that vanishes at $\rho=0$ (vacuum), and
then the corresponding pair $(\eta, q)$ is called a weak entropy pair.
Notice that $u(t, x):=\frac{m(t,x)}{\rho(t,x)}$ is not well defined on the vacuum set $\{\rho(t, x)=0\}$,
which is one of the main reasons why the weak entropy is used,
while the momentum function $m(t,x)$ itself is
well-defined to be $0$ on $\{\rho(t, x)=0\}$ almost everywhere.
We will often use $m$ and $u=\frac{m}{\rho}$ alternatively when
$\rho>0$ without ambiguity.

Any entropy solution $(\rho, m)(t,x)$ is also required (whenever available)
to satisfy the additional entropy inequality:
\begin{equation}\label{EntropyInequality}
\partial_t\eta(\rho, m)+\partial_xq(\rho, m)\leq0
\end{equation}
in the sense of distributions for any weak entropy pair $(\eta, q)(\rho, m)$ satisfying that
$\eta(\rho, m)$ is a convex function with respect to
$(\rho, m)$.

As in \cite{Chen,CLone,DCL, DiPerna, HuangWang, LPS, LPT},
the weak entropy $\eta(\rho, u)$ satisfies
\begin{equation} \label{EEq}
\begin{cases}
\partial_{\rho\rho}\eta-\frac{p'(\rho)}{\rho^2}\partial_{uu}\eta=0\qquad\,\, \mbox{for $\rho>0$},\\[1mm]
\eta|_{\rho=0}=0,
\end{cases}
\end{equation}
and the corresponding entropy flux $q(\rho, u)$ satisfies
\begin{equation*}
\partial_\rho q(\rho, u)=u\partial_\rho \eta+\frac{p'(\rho)}{\rho}\partial_u\eta,
\qquad \partial_u q(\rho, u)=\rho \partial_\rho \eta+u\partial_u\eta
\end{equation*}
for a general pressure function $p(\rho)$, including $p^{(\theta)}(\rho)$ for $\theta\geq0$.

\subsection{Isentropic case}  When $\theta>0$, the Riemann invariants of system \eqref{E-1} are
\begin{equation*}
w_j^{(\theta)} =\frac{m}{\rho}+(-1)^{j+1}\frac{\rho^{\theta}}{\theta} \qquad\mbox{for $j=1,2$}.
\end{equation*}
The weak entropy kernel of \eqref{E-1} is
\begin{equation}\label{kernel}
\chi^{(\theta)}(\rho; s-u)=a_\theta \big[(\frac{\rho^\theta}{\theta})^2-(s-u)^2\big]_+^{\frac{1-\theta}{2\theta}}
=a_\theta [(w_1-s)(s-w_2)]_+^{\frac{1-\theta}{2\theta}}
\end{equation}
with $a_\theta=\theta^{\frac{1}{\theta}}\big(\int_{-1}^1[1-\tau^2]_+^{\frac{1-\theta}{2\theta}}\,\dd\tau\big)^{-1}$,
which is a fundamental solution of \eqref{EEq}, determined by
\begin{equation*}
\begin{cases}
\partial_{\rho\rho}\chi^{(\theta)}-\rho^{\gamma-3}\partial_{uu}\chi^{(\theta)}=0 \qquad\,\, \mbox{for $\rho>0$},\\[2mm]
\chi^{(\theta)}|_{\rho=0}=0,\quad
\chi^{(\theta)}_\rho|_{\rho=0}=\delta_{u=s},
\end{cases}
\end{equation*}
where we have used the notation: $[x]_+=\max(0, x)$.
Then we have the following representation of weak entropy pairs:

\begin{lemma}
For $\theta>0$, the weak entropy pair of \eqref{E-1} can be represented as
\begin{align*}
&\eta^{(\theta)}(\rho,m; \psi)=\int_{\mathbb{R}}\chi^{(\theta)}(\rho; s-\frac{m}{\rho})\psi(s)\,\dd s,\\
&q^{(\theta)}(\rho,m;\psi) =\int_{\mathbb{R}}\big(\theta s+(1-\theta)
\frac{m}{\rho}\big)\chi^{(\theta)}(\rho; s-\frac{m}{\rho})\psi(s)\,\dd s,
\end{align*}
for any $\psi\in C^2(\mathbb{R})$.
Furthermore, $\eta^{(\theta)}(\rho,m; \psi)$ is convex if and only if $\psi$ is
convex $(${\it cf.} {\rm \cite{LPT}}$)$.
\end{lemma}

Using the standard change of variable $s=u+\frac{\rho^\theta}{\theta}\tau$, we have
\begin{eqnarray}
\eta^{(\theta)}(\rho,m; \psi)
&=&a_\theta\int_{\mathbb{R}}\big[\big(\frac{\rho^\theta}{\theta}\big)^2-(u-s)^2\big]_+^{\frac{1-\theta}{2\theta}}\psi(s)\,\dd s
\nonumber\\
&=&
\frac{\rho\int_{\mathbb{R}}\psi(u+\frac{\rho^\theta}{\theta} \tau)[1-\tau^2]_+^{\frac{1-\theta}{2\theta}}\,\dd\tau}{\int_{\mathbb{R}}[1-\tau^2]_+^{\frac{1-\theta}{2\theta}}\,\dd\tau},
\qquad\label{eta}\\
q^{(\theta)}(\rho,m;\psi)
&=&a_\theta\int_{\mathbb{R}}(\theta s+(1-\theta)\u)\big[\big(\frac{\rho^{\theta}}{\theta}\big)^{2}-(\u-s)^2\big]_+^{\frac{1-\theta}{2\theta}}
\psi(s)\,\dd s
\nonumber\\
&=&\frac{\rho\int_{\mathbb{R}}\left(u+\rho^\theta \tau\right)\psi(u+\frac{\rho^\theta}{\theta}
	\tau)[1-\tau^2]_+^{\frac{1-\theta}{2 \theta}} \dd\tau}
{\int_{\mathbb{R}}[1-\tau^2]_+^{\frac{1-\theta}{2\theta}}\,\dd\tau}.
\label{q}
\end{eqnarray}

In particular, the energy and energy flux $(\eta^{(\theta)}_*, q^{(\theta)}_*)$
may also be obtained by choosing $\psi=\frac{1}{2}s^2-\frac{1}{2\theta(2\theta+1)}$.

\begin{theorem}[Existence Theorem for $\theta>0$]\label{ExistenceTheta1}
For any $\theta>0$, there exist both a bounded entropy solution $(\rho, m)$ of
the Cauchy problem \eqref{1.5} and \eqref{1.5b} on $\mathbb{R}^2_+$ and
a corresponding
bounded dissipation measure $D^{(\theta)}\le 0$ on $\mathbb{R}_+^2\times\mathbb{R}$ such that the weak entropy kernel
$\chi^{(\theta)}(\rho; s-\frac{m}{\rho})$ satisfies the kinetic formulation{\rm :}
\begin{equation}\label{kineticf1}
\partial_t\chi^{(\theta)}+\partial_x \big((\theta s+(1-\theta)\frac{m}{\rho})\chi^{(\theta)}\big)
=\partial_{ss} D^{(\theta)}(t, x; s)
\end{equation}	
with initial data $(\rho_0,m_0)$ such that $\rho_0(x)\ge 0$ a.e. $x\in \mathbb{R}$
and $(\rho_0, \frac{m_0}{\rho_0})(x)\in L^\infty(\mathbb{R})$.
Moreover, $(\rho, m)$ satisfies that
\begin{equation}\label{2.22a}
\frac{|m(t,x)|}{\rho(t,x)}+\frac{\rho(t,x)^{\theta}-1}{\theta}
\leq \Big\|\frac{|m_0|}{\rho_0}+\frac{\rho_0^\theta-1}{\theta}\Big\|_{L^\infty(\mathbb{R})}=:w_0^{(\theta)},
\end{equation}
and $\supp\, D^{(\theta)}(t,x; \cdot)\subset [-w_0^{(\theta)}, w_0^{(\theta)}]$.
\end{theorem}

See \cite{Chen,CLone,CLtwo,DCL,DiPerna,LPS,LPT,Perthame} for the details.

\begin{remark}
The kinetic formulation \eqref{kineticf1} is equivalent
to the weak entropy inequality \eqref{EntropyInequality} for $L^\infty$ solutions.
See Chen-Perepelitsa \cite{ChenP} and Perthame \cite{Perthame}.
\end{remark}

\subsection{Isothermal case}  When $\theta=0$, the Riemann invariants are
\begin{equation*}
w_j^{(0)} =\rho e^{(-1)^{j+1}\frac{m}{\rho}} \qquad\mbox{for $j=1, 2$}.
\end{equation*}
Then typical weak entropy pairs are
\begin{equation}
\eta_\xi(\rho, m)=\rho^{\frac{1}{1-\xi^2}}e^{\frac{\xi}{1-\xi^2}\frac{m}{\rho}},
\ \ q_\xi(\rho, m)=(\frac{m}{\rho}+\xi)\eta_\xi
\qquad\,\, \mbox{for $\xi\in(-1, 1)$}.
\label{0entropy}
\end{equation}
This generates the following family of weak entropy pairs:

\begin{lemma}
When $\theta=0$, the entropy pair $(\eta^{(0)}, q^{(0)})${\rm :}
\begin{equation}\label{0eneq}
\eta^{(0)}(\rho,m; \psi)=\int_{-1}^{1}\eta_\xi(\rho, m)\psi(\xi)\,\dd\xi, \quad
q^{(0)}(\rho,m;\psi) =\int_{-1}^{1}\big(\frac{m}{\rho}+\xi\big)\eta_\xi(\rho, m)\psi(\xi)\,\dd\xi
\end{equation}
is a weak entropy pair of \eqref{E-1} for any $\psi\in L^\infty(\mathbb{R})$ with $\supp\, \psi(\cdot)\Subset (-1,1)$.
Furthermore, $\eta^{(0)}(\rho, m; \psi)$ is convex if and only if $\psi\geq0$.
 \end{lemma}

The representation formula \eqref{0eneq} is the result  in \cite[Lemma 2.1]{HuangWang}.
It has also been shown
in  \cite[Lemma 3.1]{HuangWang},
for any $\xi\in(-1, 1)$, $\eta_\xi(\rho, m)$ is strictly convex,
which implies that $\eta^{(0)}(\rho, m; \psi)$ is convex if and only if $\psi$ is nonnegative.

\begin{theorem}[Compactness Framework for $\theta=0$]\label{1framwork}
Let $(\rho^\varepsilon, m^\varepsilon)$ be a sequence of approximate solutions
of \eqref{E-1} with $\theta=0$ satisfying
\begin{equation}\label{Linftycondition}
0\leq\rho^\varepsilon(t,x)\leq C, \ \
|m^{\varepsilon}(t,x)|\leq \rho^{\varepsilon}(t,x)\big(|\ln \rho^{\varepsilon}(t,x)|+C\big)
 \qquad \mbox{{\it a.e.} $(t,x)\in \mathbb{R}_+^2$},
\end{equation}
where $C>0$ is a constant independent of $\varepsilon$. Assume that there exists a small constant $\delta>0$ such that,
for any $\xi\in (-\delta,\delta)$,
\begin{equation}\label{H-1condition}
\partial_t\eta_\xi(\rho^\varepsilon, m^\varepsilon)+\partial_x q_\xi(\rho^\varepsilon, m^\varepsilon) \quad\ \mbox{is compact in } H^{-1}_{\rm loc}
\end{equation}
for the weak entropy pairs $(\eta_\xi, q_\xi)$ defined in \eqref{0entropy}.
Then there exist both a subsequence $($still denoted$)$ $(\rho^\varepsilon, m^\varepsilon)$ and a vector function $(\rho, m)(t, x)$  such that
\begin{equation*}
(\rho^\varepsilon, m^\varepsilon)(t, x)\rightarrow(\rho, m)(t, x) \qquad\,\, \mbox{in $L^p_{\rm loc}(\mathbb{R}^2_+)$ for all $p\in[1, \infty)\,$ as $\varepsilon\to 0$}.
\end{equation*}
\end{theorem}

\begin{remark}
In {\rm \cite{HuangWang}}, Theorem {\rm \ref{1framwork}} is stated when condition \eqref{H-1condition}
holds for the case that $\delta=1$.
It can be directly checked that Theorem {\rm \ref{1framwork}} is still valid as stated above
for some small constant $\delta>0$, by following the same argument in {\rm \cite{HuangWang}}.
\end{remark}

\begin{remark}
Condition \eqref{Linftycondition} is equivalent to
\begin{equation*}\label{Linftycondition+}
0\leq w_j^{(0)}(\rho^\varepsilon, m^\varepsilon)\leq C \qquad\,\, \mbox{for $j=1, 2$},
\end{equation*}
where $C>0$ is a constant independent of $\varepsilon$.
\end{remark}

\subsection{Main Theorem}
We now state the main theorem of this paper.
Since our main focus is on the isothermal limit $\theta\to 0$,  we always assume $\theta\in (0,\theta_0]$
for some fixed $\theta_0\in (0, \infty)$
without loss of generality; for simplicity, we take any fixed $\theta_0\in (0,1)$ in our analysis from now on throughout this paper.

\begin{theorem}[Main Theorem]\label{main}
For $\theta>0$, let $(\rho^{(\theta)}, m^{(\theta)})(t,x)$ be entropy solutions
of \eqref{E-1} satisfying \eqref{kineticf1}--\eqref{2.22a}
with initial data $(\rho^{(\theta)}_0, m^{(\theta)}_0)(x)$, as constructed in Theorem {\rm \ref{ExistenceTheta1}}.
Assume that there exists a constant $w_0>0$ {\rm (}{\it e.g.}, $w_0=\sup_{0<\theta\le \theta_0}w_0^{(\theta)}${\rm )}
independent of $\theta\in (0,\theta_0]$ such that
the initial data $(\rho^{(\theta)}_0, m^{(\theta)}_0)(x)$ satisfy that
\begin{equation}
\frac{|m_0^{(\theta)}(x)|}{\rho_0^{(\theta)}(x)}+\frac{(\rho_0^{(\theta)}(x))^{\theta}-1}{\theta}\leq w_0
\qquad\,\,\, \mbox{{\it a.e.} $x\in \mathbb{R}$}.  \label{TheataUniformBound}	
\end{equation}
Then there exist both a subsequence $($still denoted$)$ $(\rho^{(\theta)}, m^{(\theta)})$ and a vector function $(\rho, m)$ such that
\begin{equation*}
(\rho^{(\theta)}, m^{(\theta)})(t, x)\rightarrow(\rho, m)(t, x)
\quad\,\,\mbox{in $L^p_{\rm loc}(\mathbb{R}^2_+)$ for all $p\in[1,\infty)\,\,$ as $\theta\to 0$},
\end{equation*}
and $(\rho, m)(t, x)$ is an entropy solution of \eqref{E-1} with $\theta=0$ satisfying
\begin{equation*}
0\leq\rho(t, x)\leq e^{w_0}, \ \ |m(t, x)|\leq \rho(t, x)\big(|\ln \rho(t, x)|+w_0\big)
\qquad\,\, \mbox{{\it a.e.} $(t,x)\in \mathbb{R}_+^2$}
\end{equation*}
and the entropy inequality  \eqref{EntropyInequality} in the sense of distributions
for any weak entropy pair $(\eta_{\xi}, q_{\xi})$ for
$\xi\in (-\sqrt{2}+1,\sqrt{2}-1)$.
In particular, for any nonnegative function $\psi(\xi)\in L^\infty(\mathbb{R})$
with ${\rm supp}\, \psi(\xi)\Subset (-\sqrt{2}+1, \sqrt{2}-1)$,
 $(\rho, m)$ satisfies
 $$
 \partial_t \eta^{(0)}(\rho,m; \psi) +\partial_x q^{(0)}(\rho,m; \psi)\le 0
 $$
 in the sense of distributions.
\end{theorem}

\section{Uniform $L^\infty$--Estimate}
In this section, we establish the uniform $L^\infty$--estimate of the entropy solutions in Theorem \ref{ExistenceTheta1}
with respect to the adiabatic exponents $\gamma>1$, {\it i.e.}, $\theta>0$.

\begin{lemma}[Uniform $L^\infty$--Bound with respect to $\theta>0$]\label{gammaluniformbound}
Let $(\rho^{(\theta)},m^{(\theta)})(t,x)$ be the entropy solutions of \eqref{E-1} with
initial data  $(\rho^{(\theta)}_0, m^{(\theta)}_0)(x)$
satisfying \eqref{TheataUniformBound}, as
constructed in Theorem {\rm \ref{ExistenceTheta1}}.
Then, for almost everywhere $(t, x)\in \mathbb{R}^2_+$,
\begin{align}\label{uniformboumd}
0\le \rho^{(\theta)}(t,x)\leq e^{w_0},\qquad |m^{(\theta)}(t,x)|\leq \rho^{(\theta)}(t,x)\big(|\ln \rho^{(\theta)}(t,x)|+ w_0\big),
\end{align}
and $	(\eta^{(\theta)}_*, q^{(\theta)}_*)(\rho^{(\theta)}, m^{(\theta)})$ are uniformly bounded with respect to
$\theta\in (0, \theta_0]$.
\end{lemma}

\begin{proof}
Using Theorem \ref{ExistenceTheta1} and condition \eqref{TheataUniformBound},
for each fixed $\theta>0$,
the entropy solution $(\rho^{(\theta)},  m^{(\theta)})$ of \eqref{E-1}
satisfies
\begin{equation}\label{3.1}
0\leq\frac{(\rho^{(\theta)}(t,x))^{\theta}}{\theta}\leq w_0 +\frac{1}{\theta}, \,\,\,\,
\frac{|m^{(\theta)}(t,x)|}{\rho^{(\theta)}(t,x)}\leq w_0 +\frac{1-(\rho^{(\theta)}(t,x))^\theta}{\theta}
\qquad \mbox{for {\it a.e.} $(t, x)\in \mathbb{R}^2_+$},
\end{equation}
which implies
\begin{equation}\label{3.2a}
0\leq\rho^{(\theta)}(t,x)\leq(1+\theta w_0)^{\frac{1}{\theta}}\leq e^{w_0} \qquad \mbox{for {\it a.e.} $(t, x)\in \mathbb{R}^2_+$}.
\end{equation}
On the other hand,
\begin{equation*}
\rho^{(\theta)}(t,x) e^{\frac{|m^{(\theta)}(t,x)|}{\rho^{(\theta)}(t,x)}}
\leq \rho^{(\theta)}(t,x) e^{-\frac{\left(\rho^{(\theta)}(t,x)\right)^\theta -1}{\theta}}e^{w_0} \qquad \mbox{for {\it a.e.} $(t, x)\in \mathbb{R}^2_+$}.
\end{equation*}
Using
\begin{equation}\label{3.3a}
\rho^{\theta}-1-\theta \ln\rho\geq 0\qquad \mbox{for any $\rho\ge 0$ and $\theta>0$},
\end{equation}
we see that
$\rho^{(\theta)} e^{-\frac{(\rho^{(\theta)})^\theta -1}{\theta}}=e^{-\frac{1}{\theta}((\rho^{(\theta)})^\theta-1-\theta\ln\rho^{(\theta)})}\leq 1$,
which leads to
\begin{equation*}
\rho^{(\theta)}(t,x) e^{\frac{|m^{(\theta)}|}{\rho^{(\theta)}}}\leq  e^{w_0} \qquad \mbox{for {\it a.e.} $(t, x)\in \mathbb{R}^2_+$},
\end{equation*}
which is uniformly bounded with respect to $\theta\in (0,\theta_0]$.

\smallskip
We now show the uniform boundedness of the energy-energy flux pair $(\eta^{(\theta)}_*, q^{(\theta)}_*)(\rho, m)$.
From \eqref{Energy}, for $\theta>0$, we have
\begin{equation}\label{IsenEnergy}
(\eta^{(\theta)}_*, q^{(\theta)}_*)(\rho, m)
=(\frac{1}{2}\frac{m^2}{\rho}+ \rho \frac{\rho^{2\theta}-1}{2\theta(2\theta+1)},
	\,
	\frac{1}{2}\frac{m^3}{\rho^2}
     +m(\frac{\rho^{2\theta}-1}{2\theta}+\frac{1}{2\theta+1})).
\end{equation}
Based on estimates \eqref{3.1}--\eqref{3.2a} and
formula \eqref{IsenEnergy},
it suffices to estimate $\frac{m^{(\theta)}}{(\rho^{(\theta)})^{\frac{2}{3}}}$.
We first obtain from \eqref{3.2a}--\eqref{3.3a} that
\begin{eqnarray*}
\frac{|m^{(\theta)}|}{(\rho^{(\theta)})^{\frac{2}{3}}}
\leq w_0(\rho^{(\theta)})^{\frac{1}{3}}+(\rho^{(\theta)})^{\frac{1}{3}}\frac{1-(\rho^{(\theta)})^\theta}{\theta}
\leq w_0e^{\frac{w_0}{3}}-(\rho^{(\theta)})^{\frac{1}{3}}\ln \rho^{(\theta)}
\leq w_0e^{\frac{w_0}{3}} + 3e^{-1}.
\end{eqnarray*}
Then there exists a universal constant $C:=C(w_0, \theta_0)>0$ independent of $\theta\in (0, \theta_0]$ such that
\begin{eqnarray*}
|m^{(\theta)}|+\frac{|m^{(\theta)}|^2}{\rho^{(\theta)}}\le C(\rho^{(\theta)})^{\frac{1}{3}}\le C,
\qquad
\big|m^{(\theta)}\frac{\rho^{2\theta}-1}{2\theta}\big|
\le C(\rho^{(\theta)})^{\frac{1}{3}}\frac{(\rho^{(\theta)})^{\theta}-1}{\theta}
\le C.
\end{eqnarray*}
Thus, the uniform boundedness of $(\eta^{(\theta)}_*, q^{(\theta)}_*)(\rho, m)$ with respect to $\theta\in (0, \theta_0]$ follows.
This completes the proof.
\end{proof}

\section{Refined Kinetic Formulation for $\theta>0$}

In this section, we show that the  total variation of the dissipation measures
in the kinetic formulation
of the entropy solutions of the isentropic Euler equations \eqref{E-1} with initial data
satisfying \eqref{TheataUniformBound} is locally uniformly bounded with respect to $\theta\in (0, \theta_0]$.
As indicated in Theorem {\rm \ref{ExistenceTheta1}},
for fixed $\theta>0$,
$(\rho^{(\theta)}, m^{(\theta)})$ is an entropy solution of \eqref{E-1} satisfying
the kinetic equation \eqref{kineticf1} with kernel $\chi^{(\theta)}$ defined in \eqref{kernel}
and the dissipation measure $D^{(\theta)}$ on $\mathbb{R}_+^2\times\mathbb{R}$.

\begin{lemma}[Uniformly Bounded Total Variation of the Dissipation Measures]
Let the initial data $(\rho^{(\theta)}_0, m^{(\theta)}_0)(x)$ satisfy \eqref{TheataUniformBound}.
Assume that the entropy solutions  $(\rho^{(\theta)}, m^{(\theta)})(t,x)$ and
the dissipation measures $D^{(\theta)}(t,x;s)$
are as constructed in Theorem {\rm \ref{ExistenceTheta1}}.
Then the total variation of the dissipation measures $D^{(\theta)}(t,x;s)$
is locally uniformly bounded with respect to
$\theta\in (0, \theta_0]${\rm :} For any compact set $K\Subset\mathbb{R}^2_+$,
there exists $C(w_0; K)>0$ such that
\begin{equation}\label{DuffEa}
\int_{\mathbb{R}} |D^{(\theta)}(\cdot,\cdot; s)|(K)\, \dd s
\leq C(w_0;K),
\end{equation}
and $\supp\, D^{(\theta)}(t,x; \cdot)\subset [-w_0, w_0]$.
\end{lemma}

\begin{proof}
For any convex function $\psi(s)$, {\it i.e.}, $\psi''(s)\ge 0$, it follows from Theorem \ref{ExistenceTheta1}
that the entropy solutions $(\rho^{(\theta)}, m^{(\theta)})$ satisfy
\begin{equation}\label{4.3}
\partial_t\eta^{(\theta)}(\rho^{(\theta)}, m^{(\theta)}; \psi)+\partial_x q^{(\theta)} (\rho^{(\theta)}, m^{(\theta)}; \psi)
=\int_\R D^{(\theta)}(t,x; s) \psi''(s)\,{\rm d} s \le 0
\end{equation}
in the sense of distributions, since
$D^{(\theta)}(t,x;s)\leq0$ on $\R_+^2\times \R$.

In \eqref{4.3}, we choose
$\psi(s)=\psi_*(s):=\frac{1}{2}s^2-\frac{1}{2\theta(2\theta+1)}$ with  $\psi_*''(s)=1$,
which leads to
\begin{equation}\label{4.3b}
\partial_t\eta^{(\theta)}_*(\rho^{(\theta)}, m^{(\theta)}; \psi)
  +\partial_x q^{(\theta)}_* (\rho^{(\theta)}, m^{(\theta)}; \psi)
=-\int_\R D^{(\theta)}(t,x; s)\,{\rm d} s
=\int_\R |D^{(\theta)}(t,x; s)|\,{\rm d} s
\end{equation}
in the sense of distributions.
Then we take the smooth cut-off function $\varphi(t,x)\ge 0$ as a test function in \eqref{4.3b}
such that $\varphi\equiv 1$ on $K$,
and $\hat{K}:=\supp\,\varphi$ is compact in $\mathbb{R}^2_+$ so that $K\subset \hat{K}\Subset\mathbb{R}^2_+$.
Then we have
\begin{eqnarray}
\int_\mathbb{R}|D^{(\theta)}(\cdot,\cdot;s)|(K)\,\dd s
&\leq&
-\int_{\mathbb{R}}\langle
  D^{(\theta)}(\cdot,\cdot; s), \phi\rangle\, \dd s
  \nonumber\\
&=&\int_{\hat{K}}
\Big(\eta^{(\theta)}_*(\rho^{(\theta)}, m^{(\theta)})\partial_t\phi
  + q^{(\theta)}_*(\rho^{(\theta)}, m^{(\theta)})\partial_x\phi\Big)\, \dd x \dd t\nonumber\\[2mm]
&\le& C(w_0;K),\nonumber
\end{eqnarray}
where we have used Lemma \ref{gammaluniformbound}, $D^{(\theta)}(t,x;s)\le 0$,
and  equation \eqref{4.3} by taking  $\varphi$ as the test function.
Therefore, we conclude that the total variation of $D^{(\theta)}(t,x;s)$ is locally uniformly
bounded with respect to $\theta\in (0,\theta_0]$ as measures, as indicated in \eqref{DuffEa}.

Moreover, we know that $\supp\, D^{(\theta)}(t,x; \cdot)\subset [-w_0^{(\theta)}, w_0^{(\theta)}]$.
Since $0\le w_0^{(\theta)} \le w_0$,
we conclude that $\supp\, D^{(\theta)}(t,x; \cdot)\subset [-w_0, w_0]$.
\end{proof}

\section{Relation between $(\eta^{(\theta)}, q^{(\theta)})$ and $(\eta_{\xi}, q_{\xi})$}

In this section,
we show that a uniform relation between
a family of weak entropy pairs $(\eta^{(\theta)}_\xi, q^{(\theta)}_\xi)$
and the corresponding entropy pair $(\eta_{\xi}, q_{\xi})$
on any function $(\rho^{(\theta)}, m^{(\theta)})$ satisfying \eqref{uniformboumd}
with respect to
$\theta\in (0, \theta_0]$.

The weak entropy pairs $(\eta^{(\theta)}_\xi, q^{(\theta)}_\xi)$
of system \eqref{E-1} with $\theta>0$
under consideration are
\begin{eqnarray}
&&\eta^{(\theta)}_{\xi}(\rho,m)
 :=\frac{\rho\int_{\R}e^{\frac{\xi}{1-\xi^2}(\frac{m}{\rho}+\frac{\rho^{\theta}}{\theta}\tau)}
  [1-\tau^2]_+^{\frac{1-\theta}{2\theta}}\dd\tau}{\int_{\R}e^{\frac{\xi\tau}{\theta(1-\xi^2)}}
   [1-\tau^2]_+^{\frac{1-\theta}{2\theta}}\dd\tau}\ge 0,
 \label{xientropy}\\
&&q^{(\theta)}_{\xi}(\rho, m)
:=\frac{\rho\int_{\R}(\frac{m}{\rho}+\rho^{\theta}\tau)e^{\frac{\xi}{1-\xi^2}( \frac{m}{\rho}+\frac{\rho^{\theta}}{\theta}\tau)}[1-\tau^2]_+^{\frac{1-\theta}{2\theta}}
  \dd\tau}{\int_{\R}e^{\frac{\xi\tau}{\theta(1-\xi^2)}}[1-\tau^2]_+^{\frac{1-\theta}{2\theta}}\dd\tau}.
\label{xientropy-2}
\end{eqnarray}
These entropy pairs are obtained by choosing $\psi(s)$ in \eqref{eta}--\eqref{q} as
\begin{equation}\label{5.3}
\psi^{(\theta)}_{\xi}(s)
=\frac{\int_{\R}[1-\tau^2]_+^{\frac{1-\theta}{2\theta}}\dd\tau}{\int_{\R}e^{\frac{\xi\tau}{\theta(1-\xi^2)}}[1-\tau^2]_+^{\frac{1-\theta}{2\theta}}\dd\tau}\,e^{\frac{\xi}{1-\xi^2}s}.
\end{equation}

We first consider the uniform boundedness of $(\eta_{\xi}^{(\theta)}, q_{\xi}^{(\theta)})$.

\begin{lemma}\label{entropyuniform}
For any function $(\rho^{(\theta)}, m^{(\theta)})$ satisfying \eqref{uniformboumd},
$(\eta_{\xi}^{(\theta)}, q_{\xi}^{(\theta)})(\rho^{(\theta)}, m^{(\theta)})$ defined in
\eqref{xientropy}--\eqref{xientropy-2}
are uniformly bounded
with respect to
$\theta\in (0,\theta_0]$
when $\xi\in [-\sqrt{2}+1, \sqrt{2}-1]$,
\end{lemma}

\begin{proof} For simplicity of notation, we drop the superscript in $(\rho^{(\theta)}, m^{(\theta)})$ in this proof.

First, $\xi\in [-\sqrt{2}+1,\sqrt{2}-1]$ is equivalent to the inequality: $\frac{2|\xi|}{1-\xi^2}\le 1$.
By definition, we have	
\begin{eqnarray}
\big|\eta^{(\theta)}_{\xi}(\rho, m)\big|
&=& \frac{\int_{-1}^{1}
	\big(\rho e^{\frac{\xi}{1-\xi^2}\frac{m}{\rho}}e^{\frac{\xi\tau}{1-\xi^2} \frac{\rho^{\theta}-1}{\theta}}\big)
    e^{\frac{\xi\tau}{\theta(1-\xi^2)}}[1-\tau^2]_+^{\frac{1-\theta}{2\theta}}\dd\tau}{\int_{-1}^{1}
      e^{\frac{\xi\tau}{\theta(1-\xi^2)}}[1-\tau^2]_+^{\frac{1-\theta}{2\theta}}\dd\tau}\nonumber\\
&\leq&\sup_{\tau\in[-1, 1]} \Big(\rho e^{\frac{\xi}{1-\xi^2}\frac{m}{\rho}}e^{\frac{\xi\tau}{1-\xi^2} \frac{\rho^{\theta}-1}{\theta}}\Big)
 \nonumber\\
&\leq& \big(\rho e^{\frac{|m|}{\rho}}\big)^{\frac{|\xi|}{1-\xi^2}}
\sup_{\tau\in[-1, 1]} \Big(\rho^{1-\frac{|\xi|}{1-\xi^2}} e^{\frac{\xi\tau}{1-\xi^2} \frac{\rho^{\theta}-1}{\theta}}\Big)\nonumber\\
&\leq&e^{\frac{|\xi|}{1-\xi^2}w_0}
\sup_{\tau\in[-1, 1]}\Big(\rho^{1-\frac{|\xi|}{1-\xi^2}} e^{\frac{\xi\tau}{1-\xi^2} \frac{\rho^{\theta}-1}{\theta}}\Big).
\nonumber
\end{eqnarray}
When $\rho\geq1$,
\begin{equation*}
\sup_{\tau\in[-1, 1]}\big(\rho^{1-\frac{|\xi|}{1-\xi^2}} e^{\frac{\xi\tau}{1-\xi^2} \frac{\rho^{\theta}-1}{\theta}}\big)
=\rho^{1-\frac{|\xi|}{1-\xi^2}} e^{\frac{|\xi|}{1-\xi^2} \frac{\rho^{\theta}-1}{\theta}}
\leq e^{\big(1-\frac{|\xi|}{1-\xi^2}\big)w_0+\frac{|\xi|}{1-\xi^2}w_0}\le e^{w_0};
\end{equation*}
when $\rho\in [0, 1)$,
\begin{equation*}
\sup_{\tau\in[-1, 1]}\big(\rho^{1-\frac{|\xi|}{1-\xi^2}} e^{\frac{\xi\tau}{1-\xi^2} \frac{\rho^{\theta}-1}{\theta}}\big)
=\rho^{1-\frac{|\xi|}{1-\xi^2}} e^{-\frac{|\xi|}{1-\xi^2} \frac{\rho^{\theta}-1}{\theta}}\leq\rho^{1-\frac{2|\xi|}{1-\xi^2}}
\leq 1 \leq e^{w_0},
\end{equation*}
where we have used \eqref{3.3a} and $\frac{\rho^{\theta}-1}{\theta}\leq w_0$.
The argument for the uniform boundedness of $q_{\xi}^{(\theta)}(\rho^{(\theta)}, m^{(\theta)})$ is similar.
\end{proof}

The main estimate of this section is

\begin{proposition}\label{lemma5.3}
Given any
$\theta>0$
and any function $(\rho^{(\theta)}, m^{(\theta)})$ satisfying \eqref{uniformboumd},
then,
for any interval $[a,b]\Subset (-\sqrt{2}+1,\sqrt{2}-1)$, there exists $C:=C(w_0; a,b)>0$
independent of $\theta\in (0,\theta_0]$ such that
\begin{eqnarray*}
\big|\eta^{(\theta)}_{\xi}(\rho^{(\theta)}, m^{(\theta)})-\eta_{\xi}(\rho^{(\theta)}, m^{(\theta)})\big|
+\big|q^{(\theta)}_{\xi}(\rho^{(\theta)}, m^{(\theta)})-q_{\xi}(\rho^{(\theta)}, m^{(\theta)})\big|
\le  C\sqrt{\theta}
\end{eqnarray*}
for any $\xi \in [a,b]$.
\end{proposition}

\begin{proof} For simplicity, we drop the superscript in $(\rho^{(\theta)}, m^{(\theta)})$ in the proof.
We divide the proof into four steps.

\medskip
1. The difference
between $\eta_\xi^{(\theta)}$ and $\eta_\xi$ is
\begin{eqnarray}
\eta^{(\theta)}_\xi(\rho, m) -\eta_{\xi}(\rho, m)
&=& \frac{\int_{-1}^1
	\rho e^{\frac{\xi}{1-\xi^2}\frac{m}{\rho}}\big(e^{\frac{\xi}{1-\xi^2} \frac{\rho^{\theta}-1}{\theta}\tau}-\rho^{\frac{\xi^2}{1-\xi^2}}\big)
	e^{\frac{\xi\tau}{\theta(1-\xi^2)}}[1-\tau^2]_+^{\frac{1-\theta}{2\theta}}\dd\tau}
  {\int_{-1}^1 e^{\frac{\xi\tau}{\theta(1-\xi^2)}}[1-\tau^2]_+^{\frac{1-\theta}{2\theta}}\dd\tau} \nonumber\\
&=:& \langle \nu_\xi^{(\theta)}(\cdot),\, f_\xi(\cdot)\rangle\nonumber\\
&=& \langle \nu_\xi^{(\theta)}(\cdot),\, f_\xi(\cdot)-f_\xi(\xi)\rangle  + f_\xi(\xi), \label{5.5a}
\end{eqnarray}
where
\begin{equation}\label{def-f}
f_\xi(\tau)
=\rho e^{\frac{\xi}{1-\xi^2}\frac{m}{\rho}}\Big(e^{\frac{\xi}{1-\xi^2} \frac{\rho^{\theta}-1}{\theta}\tau}-\rho^{\frac{\xi^2}{1-\xi^2}}\Big)
\end{equation}
is a $C^1$--function in $\tau\in [-1,1]$, and
\begin{equation*}
\nu_\xi^{(\theta)}(\tau)=  \frac{e^{\frac{\xi\tau}{\theta(1-\xi^2)}}[1-\tau^2]_+^{\frac{1-\theta}{2\theta}}}
  {\int_{-1}^1 e^{\frac{\xi\tau}{\theta(1-\xi^2)}}[1-\tau^2]_+^{\frac{1-\theta}{2\theta}}\dd\tau}
\end{equation*}	
may be regarded as a probability measure over $C([-1,1])$ for each fixed $\theta>0$.

\medskip
2. {\it Claim} 1: For any $C^1$--function $f(\tau)$ and
$\theta\in (0,\theta_0]$,
\begin{equation*}
|\langle\nu^{(\theta)}_\xi, f(\cdot)-f(\xi)\rangle|\leq C\|f'\|_{L^\infty} \sqrt{\theta}
\qquad\,\,\mbox{for any
$\xi\in (-\sqrt{2}+1,\sqrt{2}-1)$},
\end{equation*}
where $C>0$ is a constant independent of
$\theta\in (0, \theta_0]$,
$f\in C^1([-1,1])$, and
$\xi\in (-\sqrt{2}+1,\sqrt{2}-1)$.

The claim can be shown as follows:
Notice that
\begin{eqnarray}\label{5.13}
|\langle\nu^{(\theta)}_\xi, f(\cdot)-f(\xi)\rangle|
&=&\left|\frac{ \int_{\R}(f(\tau)-f(\xi)) e^{\frac{\xi}{\theta(1-\xi^2)}\tau}[1-\tau^2]_+^{\frac{1}{2\theta}}[1-\tau^2]_+^{-\frac{1}{2}}
\dd\tau}{\int_{\R}e^{\frac{\xi}{\theta(1-\xi^2)}\tau}[1-\tau^2]_+^{\frac{1}{2\theta}}[1-\tau^2]_+^{-\frac{1}{2}}\dd\tau}\right|\nonumber\\
&\le& \frac{ \int_{\R}|f(\tau)-f(\xi)|e^{\frac{\xi}{1-\xi^2}\tau}\big(e^{\frac{\xi}{1-\xi^2}\tau}
	[1-\tau^2]_+^{\frac{1}{2}}\big)^{\frac{1-\theta}{\theta}}
	\big(e^{\frac{\xi}{1-\xi^2}\xi}(1-\xi^2)^{\frac{1}{2}}\big)^{-\frac{1-\theta}{\theta}}\dd\tau}
{\int_{\R}e^{\frac{\xi}{1-\xi^2}\tau}\big(e^{\frac{\xi}{1-\xi^2}\tau}[1-\tau^2]_+^{\frac{1}{2}}\big)^{\frac{1-\theta}{\theta}}
  \big(e^{\frac{\xi}{1-\xi^2}\xi}(1-\xi^2)^{\frac{1}{2}}\big)^{-\frac{1-\theta}{\theta}}\dd\tau}\nonumber\\
&=&\frac{ \int_{-1}^{1}|f(\tau)-f(\xi)|e^{\frac{\xi}{1-\xi^2}\tau}e^{\frac{g(\tau)-g(\xi)}{h(\theta)}}\dd\tau}
{\int_{-1}^{1}e^{\frac{\xi}{1-\xi^2}\tau}e^{\frac{g(\tau)-g(\xi)}{h(\theta)}}\dd\tau},
\end{eqnarray}
where $h(\theta)=\frac{\theta}{1-\theta}>0$, and $g(\tau)=\frac{\xi}{1-\xi^2}\tau+\frac{1}{2}\ln(1-\tau^2)$
with $g'(\xi)=0$ and $g''(\tau)=-\frac{1+\tau^2}{(1-\tau^2)^2}$
which is singular near $\tau=\pm1$.
Again, using the Taylor expansion, we have
\begin{equation*}
g(\tau)-g(\xi)=\frac12g''(\tilde{\tau})(\tau-\xi)^2=-\frac{1+{\tilde{\tau}}^2}{2(1-{\tilde{\tau}}^2)^2}(\tau-\xi)^2\le -\frac12(\tau-\xi)^2
\end{equation*}
for some $\tilde{\tau}$ between $\xi$ and $\tau$.

On the other hand, when $\tau\in (-\frac12,\frac12)$, $\tilde{\tau}$ must be in $(-\frac12,\frac12)$ for any
$\xi\in (-\sqrt{2}+1,\sqrt{2}-1)$.
Then we have
\begin{equation}\label{5.14a}
g(\tau)-g(\xi)\ge -\frac{1+\frac14}{2(1-\frac14)^2}(\tau-\xi)^2=-\frac{10}{9}(\tau-\xi)^2 \qquad\,\,\, \mbox{if $\tau\in (-\frac12,\frac12)$}.
\end{equation}
Thus, we see from \eqref{5.13}--\eqref{5.14a} that
\begin{eqnarray*}
|\langle\nu^{(\theta)}_\xi, f(\cdot)-f(\xi)\rangle|
&\leq&
\frac{e^{\frac{2|\xi|}{1-\xi^2}} \int_{-1}^{1}|f(\tau)-f(\xi)| e^{-\frac{(\xi-\tau)^2}{2h(\theta)}}\dd\tau}
{\int_{-\frac{1}{2}}^{\frac{1}{2}}e^{-\frac{10(\xi-\tau)^2}{9h(\theta)}}\dd\tau}\nonumber\\
&\leq &  \frac{\int_{-\frac{1+\xi}{\sqrt{h(\theta)}}}^{\frac{1-\xi}{\sqrt{h(\theta)}}}|f(\xi+\sqrt{h(\theta)}s)-f(\xi)| e^{-\frac12s^2}\dd s}
{\int_{-\frac{1+2\xi}{2\sqrt{h(\theta)}}}^{\frac{1-2\xi}{2\sqrt{h(\theta)}}}e^{-\frac{10}{9}s^2}\dd s}\nonumber\\
&\leq & C \|f'\|_{L^\infty}\sqrt{h(\theta)}\frac{ \int_{-\infty}^{\infty} |s|e^{-\frac12s^2}\dd s}
{\int_{-\frac{3}{2}+\sqrt{2}}^{\frac{3}{2}-\sqrt{2}}e^{-\frac{10}{9}s^2}\dd s}
\leq C\|f'\|_{L^\infty} \sqrt{\theta}.\nonumber
\end{eqnarray*}
We conclude the proof of Claim 1.

\medskip
3. {\it Claim} 2:
For any interval $[a,b]\Subset (-\sqrt{2}+1,\sqrt{2}-1)$, there exists $C=C(w_0; a,b)>0$
independent of $\theta\in (0,\theta_0]$ such that, for any $(\rho, m)$ satisfying \eqref{uniformboumd},
 $f_\xi(\tau)$ defined in \eqref{def-f} satisfies that
$$
|f'_\xi (\tau)|\le C(w_0; a,b) \qquad\mbox{for any $\xi \in [a,b]$ and $\tau\in [-1,1]$}.
$$

\smallskip
This can be seen as follows:
First, we have
$$
f_\xi'(\tau)=\frac{\xi}{1-\xi^2}\frac{\rho^\theta -1}{\theta} \rho e^{\frac{\xi}{1-\xi^2}(\frac{m}{\rho}+\frac{\rho^\theta-1}{\theta}\tau)}.
$$

We consider two cases:

\smallskip
Case 1: $\rho\ge 1$.  Then
\begin{equation}\label{5.8a}
|f'_\xi(\tau)|
\le  \frac{|\xi|}{1-\xi^2}\frac{\rho^\theta-1}{\theta} \rho e^{\frac{|\xi|}{1-\xi^2}(|\frac{m}{\rho}|+\frac{\rho^\theta-1}{\theta})}
\le  \frac{|\xi|}{1-\xi^2}w_0e^{w_0} e^{\frac{|\xi|}{1-\xi^2}w_0}
\le  \frac{w_0}{2}e^{\frac{3w_0}{2}},
\end{equation}
where we have used that
\begin{equation}\label{5.8b}
\frac{|\xi|}{1-\xi^2}\le \frac{1}{2}, \quad \rho^\theta-1\le \theta w_0, \quad \rho\le e^{w_0},
\quad \frac{|m|}{\rho}+\frac{\rho^\theta-1}{\theta}\le w_0.
\end{equation}

\medskip
Case 2: $\rho\in [0,1)$.  Then, for any $\xi\in [a,b]\Subset (-\sqrt{2}+1,\sqrt{2}-1)$,
\begin{equation}\label{5.8c}
|f'_\xi(\tau)|
\le \frac{1-\rho^\theta}{2\theta} \rho
  e^{\frac{|\xi|}{1-\xi^2}(|\frac{m}{\rho}|+\frac{1-\rho^\theta}{\theta})}\\
\le  \frac{1}{2}|\ln \rho| \rho e^{\frac{|\xi|}{1-\xi^2}(w_0-2\ln \rho)}\\
\le  \frac{1}{2}e^{\frac{w_0}{2}} \rho^{1-{\frac{2|\xi|}{1-\xi^2}}}|\ln \rho|\\
\le  C(w_0;a,b)
\end{equation}
for $C(w_0;a,b)>0$ depending only on $w_0$ and $(a,b)$, but independent of $\theta\in (0, \theta_0]$,
where we have used \eqref{5.8b}
and
$1-{\frac{2|\xi|}{1-\xi^2}}\ge \delta(a,b)>0$ for some $\delta(a,b)$ depending on $(a, b)$.

Therefore, Claim 2 follows by combining estimates \eqref{5.8a}--\eqref{5.8c} together.

\smallskip
4. {\it Claim} 3:
For any $(\rho, m)$ satisfying \eqref{uniformboumd},
\begin{equation}\label{rateofI1}
|f_\xi(\xi)|\leq C\theta  \qquad \mbox{for any
$\xi\in (-\sqrt{2}+1,\sqrt{2}-1)$},
\end{equation}
where $C:=C(w_0)>0$ is a constant depending only on $w_0$, but independent of
$\theta\in (0,\theta_0]$ and
$\xi\in (-\sqrt{2}+1,\sqrt{2}-1)$.

This can be seen as follows:
Similar to Lemma \ref{entropyuniform}, we have
\begin{equation*}
|f_\xi(\xi)|
\leq (\rho e^{\frac{|m|}{\rho}})^{\frac{|\xi|}{1-\xi^2}}\rho^{1-\frac{|\xi|}{1-\xi^2}}
\Big|e^{\frac{\xi^2}{1-\xi^2} \frac{\rho^{\theta}-1}{\theta}}-\rho^{\frac{\xi^2}{1-\xi^2}}
	\Big|
\leq e^{\frac{|\xi|}{1-\xi^2}w_0}\rho^{1-\frac{|\xi|}{1-\xi^2}}
\Big|e^{\frac{\xi^2}{1-\xi^2} \frac{\rho^{\theta}-1}{\theta}}-\rho^{\frac{\xi^2}{1-\xi^2}}
	\Big|.
\end{equation*}
Then statement \eqref{rateofI1} is clear when $\rho=0$.
When $\rho>0$, we use that
$\lim\limits_{\theta\rightarrow0} e^{\beta\frac{\rho^{\theta}-1}{\theta}}=\rho^{\beta}$
for any $\beta$ to obtain
\begin{equation*}
 e^{\frac{\xi^2}{1-\xi^2} \frac{\rho^{\theta}-1}{\theta}}-\rho^{\frac{\xi^2}{1-\xi^2}}
 =\int_0^\theta \frac{\dd }{\dd s}\Big(e^{\frac{\xi^2}{1-\xi^2} \frac{\rho^{s}-1}{s}}\Big) \dd s
 =\frac{\xi^2}{1-\xi^2}\int_0^\theta e^{\frac{\xi^2}{1-\xi^2} \frac{\rho^{s}-1}{s}}
  \frac{\dd}{\dd s}\Big(\frac{\rho^s-1}{s}\Big)\,\dd s.
\end{equation*}
Then
\begin{eqnarray}
|f_{\xi}(\xi)|
&\leq&	C \frac{|\xi|^2}{1-\xi^2}
	\rho^{1-\frac{|\xi|}{1-\xi^2}} \left| \int_0^\theta e^{\frac{\xi^2}{1-\xi^2} \frac{\rho^{s}-1}{s}} \frac{\dd}{\dd s} \Big(\frac{\rho^s-1}{s}\Big)\,\dd s
	\right|\nonumber\\
	&\leq&	C
	\rho^{1-\frac{|\xi|}{1-\xi^2}} \max_{s\in(0, \theta)}\Big(e^{\frac{\xi^2}{1-\xi^2} \frac{\rho^{s}-1}{s}}\Big) 	
	\int_0^\theta \left|\frac{\dd}{\dd s} \Big(\frac{\rho^s-1}{s}\Big)\right|\,\dd s
	\nonumber\\
	&\leq&	C e^{\frac{\xi^2}{1-\xi^2}w_0}
	\rho^{1-\frac{|\xi|}{1-\xi^2}}\frac{\rho^\theta-1-\theta \ln \rho}{\theta}
    \nonumber\\
	&\leq&	C  \rho^{1-\frac{|\xi|}{1-\xi^2}}\frac{\rho^\theta-1-\theta \ln \rho}{\theta},
    \nonumber
    \end{eqnarray}
since $\rho^\theta-1-\theta \ln \rho\geq0 $,
where we have used the fact that
$$
\frac{\dd}{\dd s} \Big(\frac{\rho^s-1}{s}\Big)
=\frac{1}{s^2}\big(s \rho^s \ln \rho-\rho^s+1\big)
=\frac{1}{s^2}\big(\rho^s \ln \rho^s-\rho^s+1\big)\geq 0 \qquad\mbox{for $s\in (0,\theta]$}.
$$

On the other hand,  the Taylor expansion with respect to $\theta$ gives
\begin{equation}\label{LogConergence}
\rho^{1-\frac{|\xi|}{1-\xi^2}}\big(\rho^\theta-1-\theta \ln \rho\big)
=\frac{1}{2}\rho^{1-\frac{|\xi|}{1-\xi^2}}(\ln \rho)^2{\tilde{\theta}}^2\le C\theta^2
\qquad\,\,\mbox{for some $\tilde{\theta}\in(0, \theta)$}.
\end{equation}
Then we conclude Claim 3.

\smallskip
5. Combining Claims 1--3 above together, we see from \eqref{5.5a} that
\begin{eqnarray*}
\big|\eta^{(\theta)}_\xi(\rho, m) -\eta_{\xi}(\rho, m)\big|
&\le&  |\langle \nu_\xi^{(\theta)}(\tau),\, f_\xi(\tau)-f_\xi(\xi)\rangle|  + |f_\xi(\xi)|\nonumber\\
&\le& C\|f'\|_{L^\infty}\sqrt{\theta}+C\theta
\le C\sqrt{\theta}.
\end{eqnarray*}

The proof for the estimate that $|q^{(\theta)}_{\xi}(\rho^{(\theta)}, m^{(\theta)})-q_{\xi}(\rho^{(\theta)}, m^{(\theta)})|
\le  C\sqrt{\theta}$
follows the same argument.
This completes the proof.
\end{proof}

By the similar computation to \eqref{LogConergence},
we obtain
\begin{equation}\label{EnergyConvergenceRate}
\big|\eta^{(\theta)}_{*}(\rho^{(\theta)}, m^{(\theta)})-\eta^{(0)}_{*}(\rho^{(\theta)}, m^{(\theta)})\big|
\le  C\theta, \quad\,
\big|q^{(\theta)}_{*}(\rho^{(\theta)}, m^{(\theta)})-q^{(0)}_{*}(\rho^{(\theta)}, m^{(\theta)})\big|
\le  C\theta.
\end{equation}

\begin{remark}
From this lemma,
we can see that the restriction that
$\xi\in (-\sqrt{2}+1,\sqrt{2}-1)$
is natural
for Claims {\rm 1--2} in Steps {\rm 2--3} in the proof above.	
\end{remark}

\section{Proof of the Main Theorem }

We divide the proof of the main theorem, Theorem \ref{main}, into three steps.

\smallskip
1. For any
$\theta>0$,
assume that $(\rho^{(\theta)}, m^{(\theta)})$ is the corresponding entropy solution of \eqref{E-1},
constructed in Theorem \ref{ExistenceTheta1}.
It follows from \eqref{uniformboumd} that the solution sequence $(\rho^{(\theta)}, m^{(\theta)})$ is
uniformly bounded with respect to
$\theta\in (0,\theta_0]$,
which satisfies condition \eqref{Linftycondition} in Theorem \ref{1framwork}.

We now show that the solution sequence $(\rho^{(\theta)}, m^{(\theta)})$ satisfies
condition \eqref{H-1condition} in Theorem \ref{1framwork} for any $\delta\in (0, \sqrt{2}-1)$.
Notice that
\begin{eqnarray}\label{finalE}
\partial_t\eta_{\xi}(\rho^{(\theta)}, m^{(\theta)})+\partial_xq_{\xi}(\rho^{(\theta)}, m^{(\theta)})
=: J_1^{(\theta)}+J_2^{(\theta)},
\end{eqnarray}
with
\begin{align*}
&J_1^{(\theta)}=\partial_t\big(\eta_{\xi}(\rho^{(\theta)}, m^{(\theta)})-\eta^{(\theta)}_{\xi}(\rho^{(\theta)}, m^{(\theta)})\big)
 +\partial_x\big(q_{\xi}(\rho^{(\theta)}, m^{(\theta)})-q^{(\theta)}_{\xi}(\rho^{(\theta)}, m^{(\theta)})\big),\\
&J_2^{(\theta)}=\int_{-w_0}^{w_0}(\psi_{\xi}^{(\theta)})''(s)\,D^{(\theta)}(t,x;s)\,\dd s.
\end{align*}
By Proposition \ref{lemma5.3}, we have
$$
\|J_1^{(\theta)}\|_{H^{-1}(K)}\le C_K\sqrt{\theta}\to 0 \qquad \mbox{as $\theta\to 0$}
$$
for any subset $K\Subset \mathbb{R}_+^2$. This implies that $J_1^{(\theta)}$ is compact in $H^{-1}_{\rm loc}$.

Notice that $\mbox{supp}\, D^{(\theta)}(t, x;\,\cdot)\subset [-w_0, w_0]$. By \eqref{5.3},
we see that, for $s\in [-w_0, w_0]$,
 \begin{equation*}
 (\psi_{\xi}^{(\theta)})''
 =\frac{\xi^2}{(1-\xi^2)^2}\frac{e^{\frac{\xi}{1-\xi^2}s}\int_{-1}^{1}(1-\tau^2)^{\frac{1-\theta}{2\theta}}\,
  \dd\tau}{\int_{-1}^{1}e^{\frac{\xi\tau}{\theta(1-\xi^2)}}(1-\tau^2)^{\frac{1-\theta}{2\theta}}\dd\tau}
  \le \frac{2\xi^2}{(1-\xi^2)^2}e^{\frac{|\xi|w_0}{1-\xi^2}}\le C,
 \end{equation*}
where we have used the fact that $e^{\frac{\xi\tau}{\theta(1-\xi^2)}}\ge 1$ when $\xi\tau\ge 0$.
Thus, $J_2^{(\theta)}$ is uniformly bounded as a Radon measure sequence,
which implies that $J_2^{(\theta)}$ is compact in $W^{-1, q}_{\rm loc}$ for $q\in (1,2)$.

Combining the compactness of both $J_1^{(\theta)}$ and $J_2^{(\theta)}$ above together,
we conclude that
$$
J^{(\theta)}_1+J^{(\theta)}_2 \qquad\, \mbox{is compact in $W^{-1, q}_{\rm loc}$ for $q\in (1,2)$}.
$$
On the other hand,
$$
J^{(\theta)}_1+J^{(\theta)}_2
 =\partial_t\eta_{\xi}(\rho^{(\theta)}, m^{(\theta)})+\partial_xq_{\xi}(\rho^{(\theta)}, m^{(\theta)})
 \qquad\mbox{is uniformly bounded in $W^{-1,\infty}_{\rm loc}$}.
$$
Then the interpolation embedding compactness ({\it cf}. \cite{DCL}) implies that  $J^{(\theta)}_1+J^{(\theta)}_2$
is the $H^{-1}_{\rm loc}$--compact.
Therefore, the solution sequence $(\rho^{(\theta)}, m^{(\theta)})$ satisfies condition \eqref{H-1condition}.

We now employ Theorem \ref{1framwork} to conclude that $(\rho^{(\theta)}, m^{(\theta)})(t,x)$
(extracting a subsequence if necessary)
strongly converges to a function $(\rho, m)(t,x)\in L^\infty$ {\it a.e.} $(t,x)\in \mathbb{R}_+^2$.
From the uniform estimates in \S 3,  we conclude
\begin{equation*}
0\leq\rho(t,x)\leq e^{w_0}, \ \ |m(t,x)|\le \rho(t,x)\big(|\ln \rho(t,x)|+ w_0\big)
\qquad \mbox{{\it a.e.} $(t,x)\in \mathbb{R}_+^2$}.
\end{equation*}

\smallskip
2. We next prove that $(\rho, m)(t,x)$ satisfies the energy inequality.
For any fixed
$\theta>0$,
$(\rho^{(\theta)}, m^{(\theta)})$ satisfies
\begin{equation}\label{theta-energy-inequ-b}
\partial_t {\eta}^{(\theta)}_*(\rho^{(\theta)}, m^{(\theta)})+\partial_x{q}^{(\theta)}_*(\rho^{(\theta)}, m^{(\theta)})
\leq0
\end{equation}
in the sense of distributions.
Notice that \eqref{EnergyConvergenceRate} holds
for any $(\rho, m)$ satisfying \eqref{uniformboumd}.
Taking $\theta\rightarrow0$ in \eqref{theta-energy-inequ-b}, we have
\begin{equation*}
\partial_t \eta^{(0)}_*(\rho, m)+\partial_x q^{(0)}_*(\rho, m)\leq 0
\end{equation*}
in the sense of distributions.
This shows that $(\rho, m)$ is an entropy solutions of system \eqref{E-1} with $\theta=0$ in the sense of Definition 2.1.

\smallskip
3. For
$\xi\in (-\sqrt{2}+1,\sqrt{2}-1)$,
there exists a non-positive measure $D(t, x; \xi)$ such that
\begin{equation*}
\int_{\mathbb{R}}(\psi_{\xi}^{(\theta)})''(s)\,D^{(\theta)}(t,x;s)\,\dd s\,\rightharpoonup\, D(t, x; \xi) \qquad \mbox{in Radon measures as $\theta\rightarrow0$}.
\end{equation*}
Taking the weak limit in \eqref{finalE}, we conclude that, for each
 $\xi\in (-\sqrt{2}+1,\sqrt{2}-1)$,
\begin{equation*}
\partial_t\eta_{\xi}(\rho, m)+\partial_xq_{\xi}(\rho, m)= D(t, x;\xi)\leq 0
\end{equation*}
in the sense of distributions.
This implies that $(\rho, m)$ satisfies the entropy inequality \eqref{EntropyInequality} with entropy pairs $(\eta_{\xi}, q_{\xi})$
in the sense of distributions.

 In particular, for any nonnegative $\psi(\xi)\in L^\infty(\mathbb{R})$ with ${\rm supp}\, \psi(\xi)\Subset [-\sqrt{2}+1,\sqrt{2}-1]$,
 $(\rho, m)$ satisfies
 $$
 \partial_t \eta^{(0)}(\rho,m; \psi) +\partial_x q^{(0)}(\rho,m; \psi)\le 0
 $$
 in the sense of distributions.

This completes the proof of Theorem \ref{main}.

\section{Convergence of Riemann Solutions Containing the Vacuum States }

In this section, we show the pointwise convergence of Riemann solutions,
which includes the vanishing vacuum states between the two rarefaction waves and
the limit of one-side vacuum states.
Consider the Riemann solutions of system \eqref{E-1} for $\theta\geq 0$ with initial data:
\begin{equation*}
(\rho_0, \frac{m_0}{\rho_0})
=\begin{cases}
(\rho_{L}, u_L)\qquad \mbox{for $x<0$},\\[1mm]
(\rho_{R}, u_R)\qquad \mbox{for $x>0$}.
\end{cases}
\end{equation*}

\begin{lemma}  For $\theta>0$, the shock curves are
\begin{eqnarray}
&& S^{\left(\theta\right)}_{1}: \,\,\,
u_{R}= u_{L}-\sqrt{\Big(\frac{1}{\rho_{L}}-\frac{1}{\rho_{R}}\Big)
 \frac{\rho^{2\theta+1}_{R}-\rho^{2\theta+1}_{L}}{2\theta+1}} \qquad\,\, \mbox{for $u_{L}> u_{R}$ and $\rho_{L}<\rho_{R}$},\label{S1T}\\[1mm]
&& S^{\left(\theta\right)}_{2}: \,\,\, u_{R}= u_{L}-\sqrt{\Big(\frac{1}{\rho_{L}}-\frac{1}{\rho_{R}}\Big)
  \frac{\rho^{2\theta+1}_{R}-\rho^{2\theta+1}_{L}}{2\theta+1}}\qquad\,\, \mbox{for $u_{L}> u_{R}$ and $\rho_{L}>\rho_{R}$},\label{S2T}
\end{eqnarray}
and the respective solutions of shocks $S^{(\theta)}_{j}$, $j=1,2$,
are
\begin{equation*}
(\rho^{(\theta)}, u^{(\theta)})
=\begin{cases}
(\rho_{L}, u_L)\qquad \mbox{for $x<\sigma t$},\\[1mm]
(\rho_{R}, u_R)\qquad \mbox{for $x>\sigma t$},
\end{cases}
\end{equation*}
with shock speed $\sigma:=\frac{[\rho u]}{[\rho]}=\frac{\rho_R u_R-\rho_L u_L}{\rho_R-\rho_L}$.
Correspondingly, the rarefaction curves are
\begin{eqnarray}
&&R^{\left(\theta\right)}_{1}: \,\,\,
  w^{(\theta)}_1(\rho_R, u_R )= w^{(\theta)}_1(\rho_L, u_L )=:w^{(\theta)}_1
   \qquad\,\, \mbox{for $u_{L}< u_{R}$ and $\rho_{L}>\rho_{R}$},\label{R1T}\\
&&R^{\left(\theta\right)}_{2}: \,\,\,
  w^{(\theta)}_2(\rho_R, u_R )= w^{(\theta)}_2(\rho_L, u_L )=:w^{(\theta)}_2
   \qquad\,\, \mbox{for $u_{L}< u_{R}$ and $\rho_{L}<\rho_{R}$},\label{R2T}
\end{eqnarray}
where
the respective solutions of rarefaction waves $R^{(\theta)}_j$, $j=1,2$, are
\begin{equation*}
(\rho^{(\theta)}, u^{(\theta)})
=\begin{cases}
(\rho_{L}, u_L)\qquad\, &\mbox{for $x\leq\lambda^{(\theta)}_j(\rho_L, u_L)t$},\\[1mm]
( (\frac{(-1)^j\theta(\frac{x}{t}-w^{(\theta)}_j)}{\theta+1})^{\frac{1}{\theta}},
  \frac{\frac{x}{t}+\theta w^{(\theta)}_j}{\theta+1})\qquad\, &\mbox{for $\lambda^{(\theta)}_j(\rho_L, u_L)t< x <\lambda^{(\theta)}_j(\rho_R, u_R)t$},\\[1mm]
(\rho_{R}, u_R)\qquad\, &\mbox{for $x\geq \lambda^{(\theta)}_j(\rho_R, u_R)t$},
\end{cases}
\end{equation*}
where  $\lambda^{(\theta)}_j(\rho, u)=u+(-1)^j\rho^{\theta}$, $j=1, 2$, are two eigenvalues.

\smallskip
For $\theta=0$, the shock curves are
\begin{eqnarray}
&&S^{\left(0\right)}_{1}: \,\,\,
  u_{R}= u_{L}-\sqrt{\Big(\frac{1}{\rho_{L}}-\frac{1}{\rho_{R}}\Big)\left(\rho_{R}-\rho_{L}\right)}
  \qquad\,\, \mbox{for $u_{L}> u_{R}$ and $\rho_{L}<\rho_{R}$},\label{S10}\\[1mm]
&&S^{\left(0\right)}_{2}: \,\,\,
 u_{R}= u_{L}-\sqrt{\Big(\frac{1}{\rho_{L}}-\frac{1}{\rho_{R}}\Big)\left(\rho_{R}-\rho_{L}\right)}
 \qquad\,\, \mbox{for $u_{L}> u_{R}$ and $\rho_{L}>\rho_{R}$},\label{S20}
\end{eqnarray}
and the respective solutions of shocks $S^{\left(0\right)}_{j}$, $j=1,2$,
are
\begin{equation*}
(\rho^{(0)}, u^{(0)})
=\begin{cases}
(\rho_{L}, u_L)\qquad \mbox{for $x<\sigma t$},\\[1mm]
(\rho_{R}, u_R)\qquad \mbox{for $x>\sigma t$},
\end{cases}
\end{equation*}
with shock speed $\sigma:=\frac{[\rho u]}{[\rho]}=\frac{\rho_R u_R-\rho_L u_L}{\rho_R-\rho_L}$.
Correspondingly, the rarefaction curves are
\begin{eqnarray}
&&R^{\left(0\right)}_{1}: \,\,\,
  w^{(0)}_1(\rho_R, u_R )=w^{(0)}_1(\rho_L, u_L )=:w^{(0)}_1\qquad\,\mbox{for $u_{L}< u_{R}$ and $\rho_{L}>\rho_{R}$},\label{R10}\\[1mm]
&&R^{\left(0\right)}_{2}: \,\,\, w^{(0)}_2(\rho_R, u_R )=w^{(0)}_2(\rho_L, u_L )=:w^{(0)}_2
 \qquad\,\mbox{for $u_{L}< u_{R}$ and $\rho_{L}<\rho_{R}$},\label{R20}
\end{eqnarray}
where
the respective solutions of rarefaction waves $R^{(\theta)}_j$  are{\rm :}
\begin{equation*}
(\rho^{(0)}, u^{(0)})
=\begin{cases}
(\rho_{L}, u_L)\qquad\, &\mbox{for $x\leq\lambda^{(0)}_j(\rho_L, u_L)t$},\\[1mm]
(w^{(0)}_j e^{(-1)^j\frac{x}{t}-1}, \frac{x}{t}-(-1)^j)\qquad\, &\mbox{for $\lambda^{(0)}_j(\rho_L, u_L)t< x <\lambda^{(0)}_j(\rho_R, u_R)t$},\\[1mm]
(\rho_{R}, u_R)\qquad\, &\mbox{for $x\geq \lambda^{(0)}_j(\rho_R, u_R)t$},
	\end{cases}
\end{equation*}
where $\lambda^{(0)}_j(\rho, u)=u+(-1)^j$, $j=1, 2$, are two eignevalues.
\end{lemma}

From \eqref{S1T}--\eqref{S2T} and \eqref{S10}--\eqref{S20}, if $\rho_{L}=0$ or $\rho_{R}=0$,
then $\rho_{L}=\rho_{R}=0$, which indicates that,
for $\theta\geq0$, shock waves can not connect with a vacuum state.
However, for $\theta\geq 0$,
\eqref{R1T} and \eqref{R10} show that the right-states of $R^{(\theta)}_{1}$ and $R^{(0)}_{1}$
may be a vacuum state $\rho_{R}=0$, while \eqref{R2T} and \eqref{R20} show the left-states
of $R^{(\theta)}_{2}$ and $R^{(0)}_{2}$ may also be a vacuum state $\rho_{L}=0$.
Based on the above analysis, we now separate our discussion to the two-side non-vacuum case
and the one-side vacuum case.
For the two-side non-vacuum case, we have

\begin{proposition}
For fixed
$(\rho_{L}, u_{L})$ and $(\rho_{R}, u_{R})$ with $\rho_{L}>0$ and $\rho_{R}>0$,
there exists an unique Riemann solution $(\rho^{(\theta)}, u^{(\theta)})$ for $\theta\geq0$,
and  $(\rho^{(\theta)}, u^{(\theta)})\to (\rho^{(0)} ,u^{(0)})$ as $\theta\rightarrow 0$.
\end{proposition}

\begin{proof}
We divide the proof into four steps.
		\begin{figure}
		\centering
		\includegraphics[width=0.33\linewidth]{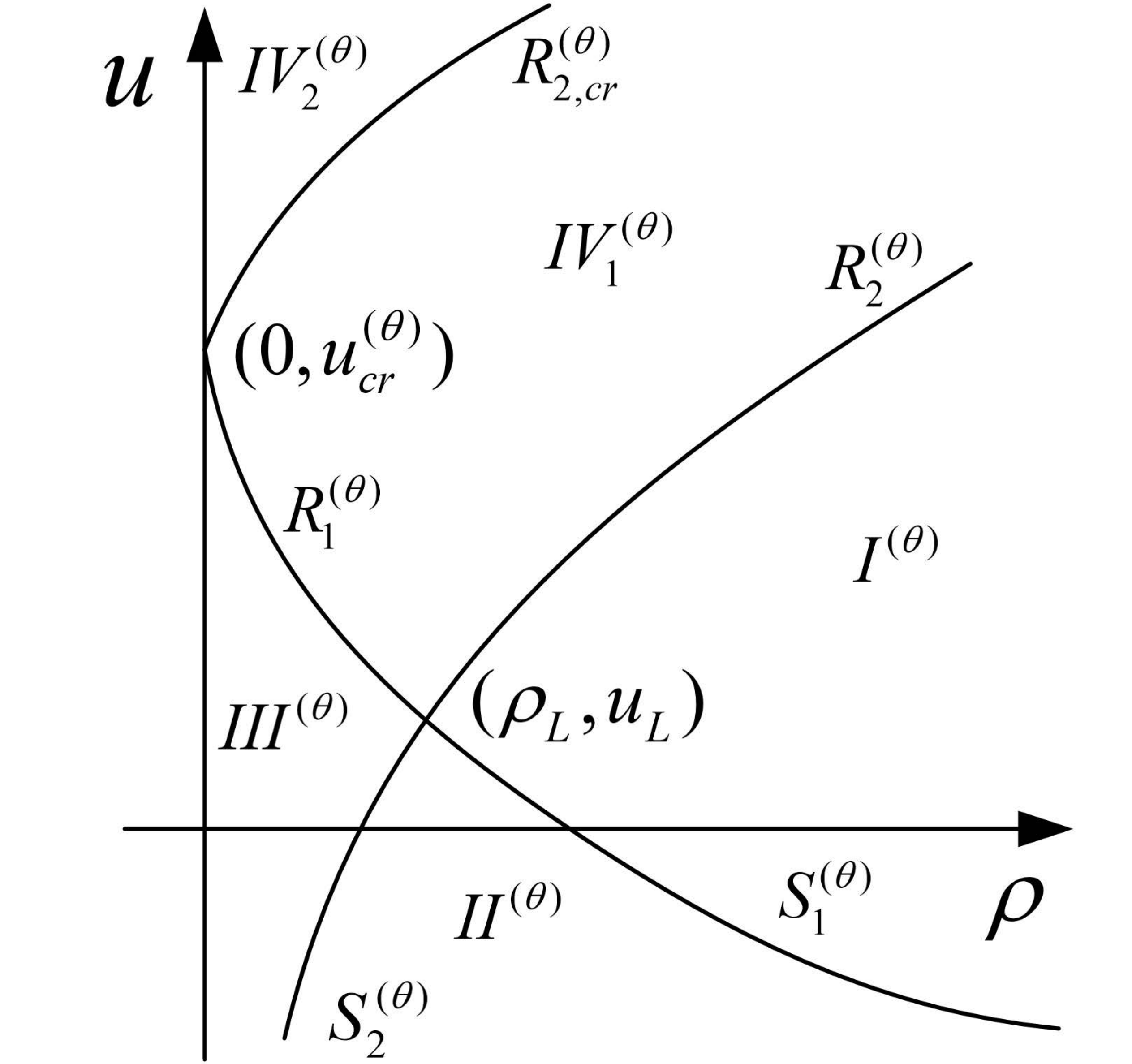}
		\caption{}
		\label{fig-rstheta}
	\end{figure}

\smallskip
1. For $\theta>0$, by \cite{Dafermos}, we have the following phase plane for fixed $(\rho_L, u_L)$
as shown in Fig. \ref{fig-rstheta},
on which $(0, u_{cr}^{(\theta)})$ with $u_{cr}^{(\theta)}=u_L+\frac{\rho_L^\theta}{\theta}$
is the intersect point of $R^{(\theta)}_1$ and $\{\rho=0\}$,
and $R^{(\theta)}_{2, cr}$ is the $2$-rarefaction curve starting from $(0, u_{cr}^{(\theta)})$
separating
region $IV_2^{(\theta)}$ from $IV_1^{(\theta)}$.
\begin{itemize}
\item If  $(\rho_{R}, u_{R})\in I^{(\theta)}$, then solution $(\rho^{(\theta)}, u^{(\theta)})$ is $S^{(\theta)}_1+R^{(\theta)}_2$.

\vspace{2pt}
\item If  $(\rho_{R}, u_{R})\in II^{(\theta)}$, then solution $(\rho^{(\theta)}, u^{(\theta)})$ is $S^{(\theta)}_1+S^{(\theta)}_2$.

\vspace{2pt}
\item If  $(\rho_{R}, u_{R})\in III^{(\theta)}$, then solution $(\rho^{(\theta)}, u^{(\theta)})$ is $R^{(\theta)}_1+S^{(\theta)}_2$.

\vspace{2pt}
\item If  $(\rho_{R}, u_{R})\in IV_1^{(\theta)}$, then solution $(\rho^{(\theta)}, u^{(\theta)})$ is $R^{(\theta)}_1+R^{(\theta)}_2$
with the middle state $\rho_M^{(\theta)}>0$.

\vspace{2pt}
\item If  $(\rho_{R}, u_{R})\in IV_2^{(\theta)}$, then solution $(\rho^{(\theta)}, u^{(\theta)})$
 is $R^{(\theta)}_1+R^{(\theta)}_2$ with the middle state $\rho_M^{(\theta)}=0$.
\end{itemize}

The criterion for the last case is $\rho_M^{\left(\theta\right)}=0$
if and only if
$\xi_{1}^{(\theta)}:=u_{L}+\frac{\rho^{\theta}_{L}}{\theta}\leq u_{R}-\frac{\rho^{\theta}_{R}}{\theta}=:\xi_{2}^{(\theta)}$.

\begin{figure}
	\centering
	\includegraphics[width=0.3\linewidth]{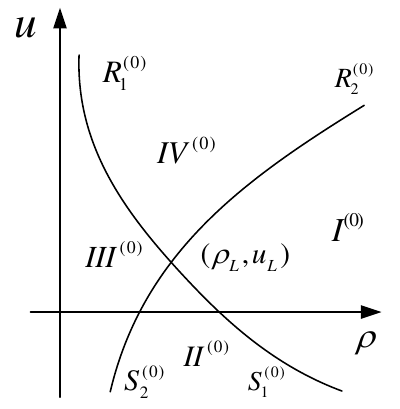}
	\caption{}
	\label{fig-rs0}
\end{figure}

\medskip
2. For $\theta=0$, we have the following phase plane for fixed $(\rho_L, u_L)$ as shown in  Fig. \ref{fig-rs0} similarly,
on which $R^{(0)}_1$ and $\{\rho=0\}$ do not intersect.
\begin{itemize}
\item If  $(\rho_{R}, u_{R})\in I^{(0)}$, then solution $(\rho^{(0)}, u^{(0)})$ is $S^{(0)}_1+R^{(0)}_2$.

\vspace{2pt}
\item If  $(\rho_{R}, u_{R})\in II^{(0)}$, then solution $(\rho^{(0)}, u^{(0)})$ is $S^{(0)}_1+S^{(0)}_2$.

\vspace{2pt}
\item If  $(\rho_{R}, u_{R})\in III^{(0)}$, then solution $(\rho^{(0)}, u^{(0)})$ is $R^{(0)}_1+S^{(0)}_2$.

\vspace{2pt}
\item If  $(\rho_{R}, u_{R})\in IV^{(0)}$, then solution $(\rho^{(0)}, u^{(0)})$ is $R^{(0)}_1+R^{(\theta)}_2$.
\end{itemize}

Then we have the existence and uniqueness of the Riemann solution  $(\rho^{(\theta)}, u^{(\theta)})$ for all $\theta\geq0$.

\smallskip
3. For the shock curves, a direct computation yields that
$(S^{(\theta)}_1, S^{(\theta)}_2)\rightarrow (S^{(0)}_1, S^{(0)}_2)$
and the shock speed $\sigma^{(\theta)}\rightarrow\sigma^{(0)}$
uniformly as $\theta\rightarrow 0$,
since $\rho_{L}$ and $\rho_{R}>0$.
Similar to the computation of \eqref{LogConergence},
for the case that $\rho_{L}>0$ and $\rho_{R}>0$,
$(R^{(\theta)}_1, R^{(\theta)}_2)\rightarrow (R^{(0)}_1, R^{(0)}_2)$ and
\begin{equation*}
(\lambda^{(\theta)}_1, \lambda^{(\theta)}_2)
=(u-\rho^{\theta}, u+\rho^{\theta})\,\rightarrow\,(\lambda^{(0)}_1, \lambda^{(0)}_2)=(u-1, u+1)
\qquad\mbox{uniformly as $\theta\rightarrow 0$}.
\end{equation*}
Thus, we conclude that the shock and rarefaction wave curves converge strongly,
for the two-side non-vacuum case.

\medskip
4. With the convergence of each wave curve, we turn to the convergence of
the corresponding Riemann solutions $(\rho^{(\theta)}, u^{(\theta)})$.
The argument is divided into two cases:

\smallskip
Case 1:  $\rho^{(\theta)}_{M}>0$ for all $\theta>0$.
In this case, for the fixed left-state $(\rho_L, u_L)$, the right-state $(\rho_R, u_R)$ belongs
to $\bigcap_{\theta\in(0, \theta_0]}\big(I^{(\theta)}\cup II^{(\theta)}\cup III^{(\theta)}\cup IV^{(\theta)}_1 \big)$.
The conclusion of Step 3 leads to the uniform convergence of $(\rho^{(\theta)}, u^{(\theta)})\to (\rho^{(0)}, u^{(0)})$
as $\theta\rightarrow 0$.
	
Case 2: There exists $\theta_{*}$ such that $\rho^{(\theta_{*})}_{M}=0$.
For this case, $(\rho_{R}, u_{R})\in IV_2^{(\theta_{*})}$,
then solution $(\rho^{(\theta_{*})}, u^{(\theta_{*})})$ is $R^{(\theta_{*})}_1+R^{(\theta_{*})}_2$ with
$\xi_{1}^{(\theta_{*})}\leq \xi_{2}^{(\theta_{*})}$.
Notice that
\begin{equation*}
\xi_{2}^{(\theta)}-\xi_{1}^{(\theta)}
=\big(u_{R}-\frac{\rho^{\theta}_{R}}{\theta}\big)-\big(u_{L}+\frac{\rho^{\theta}_{L}}{\theta}\big)
=(u_{R}-u_{L})-\frac{\rho^{\theta}_{L}+\rho^{\theta}_{R}}{\theta}.
\end{equation*}
There exists $\underline{\theta}_{*}< \theta_*$ such that
$\xi_{2}^{(\theta)}<\xi_{1}^{(\theta)}$ for $\theta<\underline{\theta}_{*}$,
which leads to $\rho^{(\theta)}_{M}>0$.
The rest of the argument is same as Case 1.
\end{proof}

\begin{remark}
From Case $2$, we observe that the centered vacuum state with $\rho^{(\theta)}_{M}=0$ vanishes when $\theta$ is small enough. This represents
the phenomenon of decavitation as $\gamma\to 1$,
which is different from the formation of cavitation and concentration in the vanishing pressure limit $($equivalently, the high Mach limit$)$ in
presented in Chen-Liu \cite{ChenLiu}.
\end{remark}
	
\medskip
Next, we consider the one-side vacuum case. Without loss of generality,
we assume that $\rho_{L}=0$, which connects $R^{(\theta)}_{2}$ with $R^{(0)}_{2}$.
Notice that the fluid velocity  $u_{L}$ is not well-defined in the vacuum with $\rho_{L}=0$ in general.
In the following analysis, we fix the right-state $(\rho_R, u_R)$.

\begin{figure}
	\centering
	\includegraphics[width=0.35\linewidth]{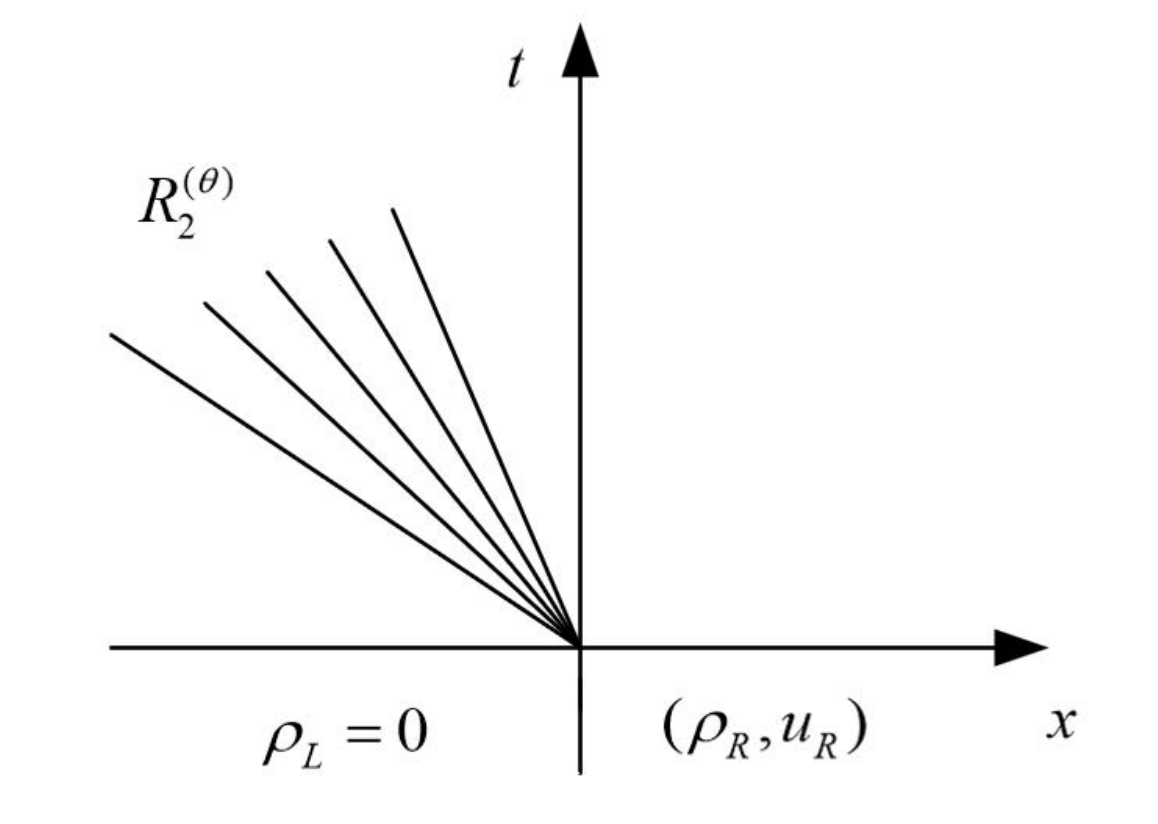}
	\caption{}
	\label{fig:vacuumtheta}
\end{figure}

For $\theta>0$, the Riemann solution $(\rho^{(\theta)}, u^{(\theta)})$ as shown in Fig. \ref{fig:vacuumtheta}
is
\begin{equation*}
(\rho^{(\theta)}, u^{(\theta)})
=\begin{cases}
((\frac{\theta\frac{x}{t}-\theta u_L+\rho_R^\theta}{\theta+1})^{\frac{1}{\theta}},
\frac{\frac{x}{t}+\theta u_L-\rho_R^\theta}{\theta+1})\qquad\,&\mbox{for $u_{R}-\frac{\rho^{\theta}_{R}}{\theta}< \frac{x}{t}< u_R+\rho_R^\theta$},\\[1mm]
(\rho_{R}, u_R)\qquad\,&\mbox{for $\frac{x}{t}\geq u_R+\rho_R^\theta$},
\end{cases}
\end{equation*}
and $\rho^{(\theta)}=0$ when $\frac{x}{t}\leq u_{R}-\frac{\rho^{\theta}_{R}}{\theta}$.
Moreover, $u^{(\theta)}(t, x)=u_{R}-\frac{\rho^{\theta}_{R}}{\theta}$
at $\frac{x}{t}=u_{R}-\frac{\rho^{\theta}_{R}}{\theta}+0$.

\begin{figure}
	\centering
	\includegraphics[width=0.35\linewidth]{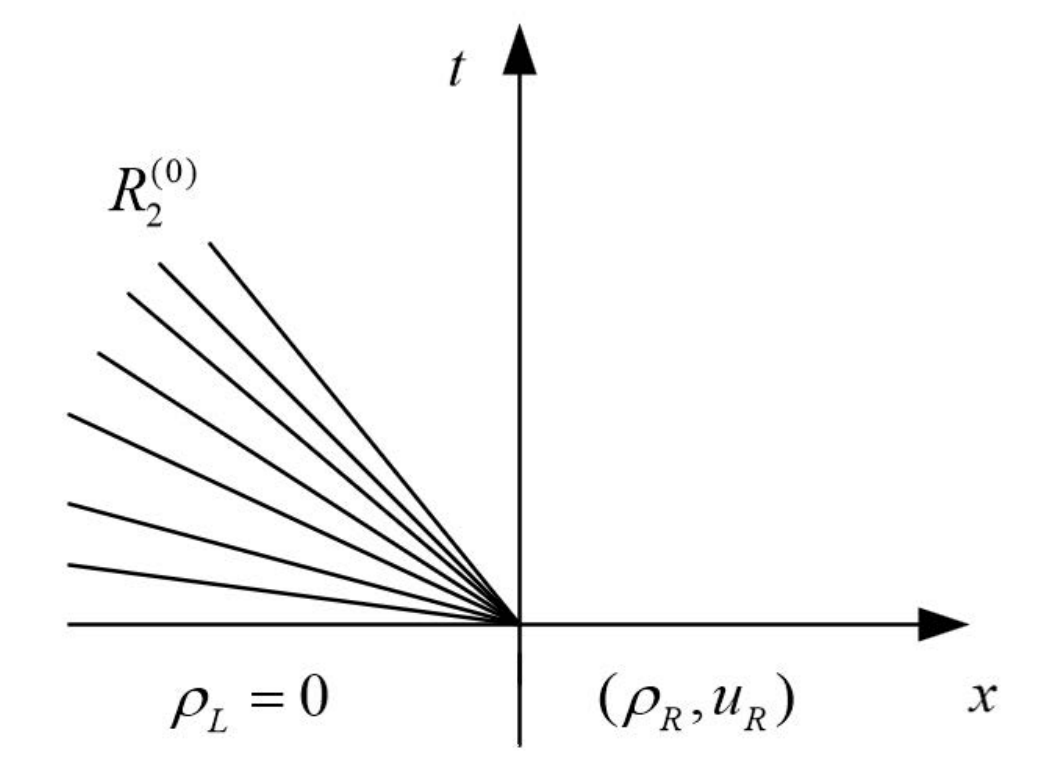}
	\caption{}
	\label{fig:vacuum0}
\end{figure}

For $\theta=0$, the Riemann solution $(\rho^{(0)}, u^{(0)})$ as shown in Fig. \ref{fig:vacuum0} is
\begin{equation*}
(\rho^{(0)}, u^{(0)})
=\begin{cases}
(\rho_{L}e^{\frac{x}{t}-u_L-1}, \frac{x}{t}-1)\qquad\,&\mbox{for $\frac{x}{t}< u_R+1$},\\[1mm]
(\rho_{R}, u_R)\qquad\,&\mbox{for $\frac{x}{t}\geq u_R+1$}.
\end{cases}
\end{equation*}
Moreover, $u^{(0)}(t, x)\rightarrow -\infty$ as $\frac{x}{t}\rightarrow -\infty$.

For $\theta\rightarrow 0$, it is direct to see that
$u_{R}-\frac{\rho^{\theta}_{R}}{\theta}\rightarrow-\infty$ and $u_{R}+\rho^{\theta}_{R}\rightarrow u_R+1$,
which leads to $(\rho^{(\theta)}, u^{(\theta)})\rightarrow	(\rho^{(0)}, u^{(0)})$.

Notice that
\begin{equation*}
\lim_{\theta\rightarrow0}\lim_{\frac{x}{t}\rightarrow u_{R}-\frac{\rho^{\theta}_{R}}{\theta}+0} c^{(\theta)}(\rho^{(\theta)}(t, x))=0,
\qquad
\lim_{\theta\rightarrow0} c^{(\theta)}(\rho^{(\theta)}(t, x))\equiv1,
\end{equation*}
which indicate the order of the limits on the sound speed $c^{(\theta)}$ close to the vacuum can not change in general.

The above case shows the convergence of the one-side vacuum states.
However, $u_L$ is not well-defined in the vacuum with $\rho_{L}=0$, so that \eqref{TheataUniformBound}
is not well-defined as well.

\smallskip
Finally, we construct a family of Riemann solutions $(\rho^{(\theta)}_A, u^{(\theta)}_A)$ approaching $(\rho^{(0)}, u^{(0)})$ with
the initial data satisfying \eqref{TheataUniformBound}. Choose the initial data:
\begin{equation*}
(\rho^{(\theta)}_A, u^{(\theta)}_A)(0, x)
=\begin{cases}
(\rho^{(\theta)}_L, u^{(\theta)}_L)\qquad\,&\mbox{for $x<0$},\\
(\rho_{R}, u_R)\qquad\,&\mbox{for $x>0$},
\end{cases}
\end{equation*}
with $\rho^{(\theta)}_L=O(\theta^{\frac{1}{\theta}})$ and
$u^{(\theta)}_L=u_R-\frac{\rho_{R}^\theta-(\rho^{(\theta)}_L)^\theta}{\theta}$.
It is direct to check that  $(\rho^{(\theta)}_A, u^{(\theta)}_A)(0, x)$
satisfy \eqref{TheataUniformBound}.

\begin{figure}
	\centering
	\includegraphics[width=0.35\linewidth]{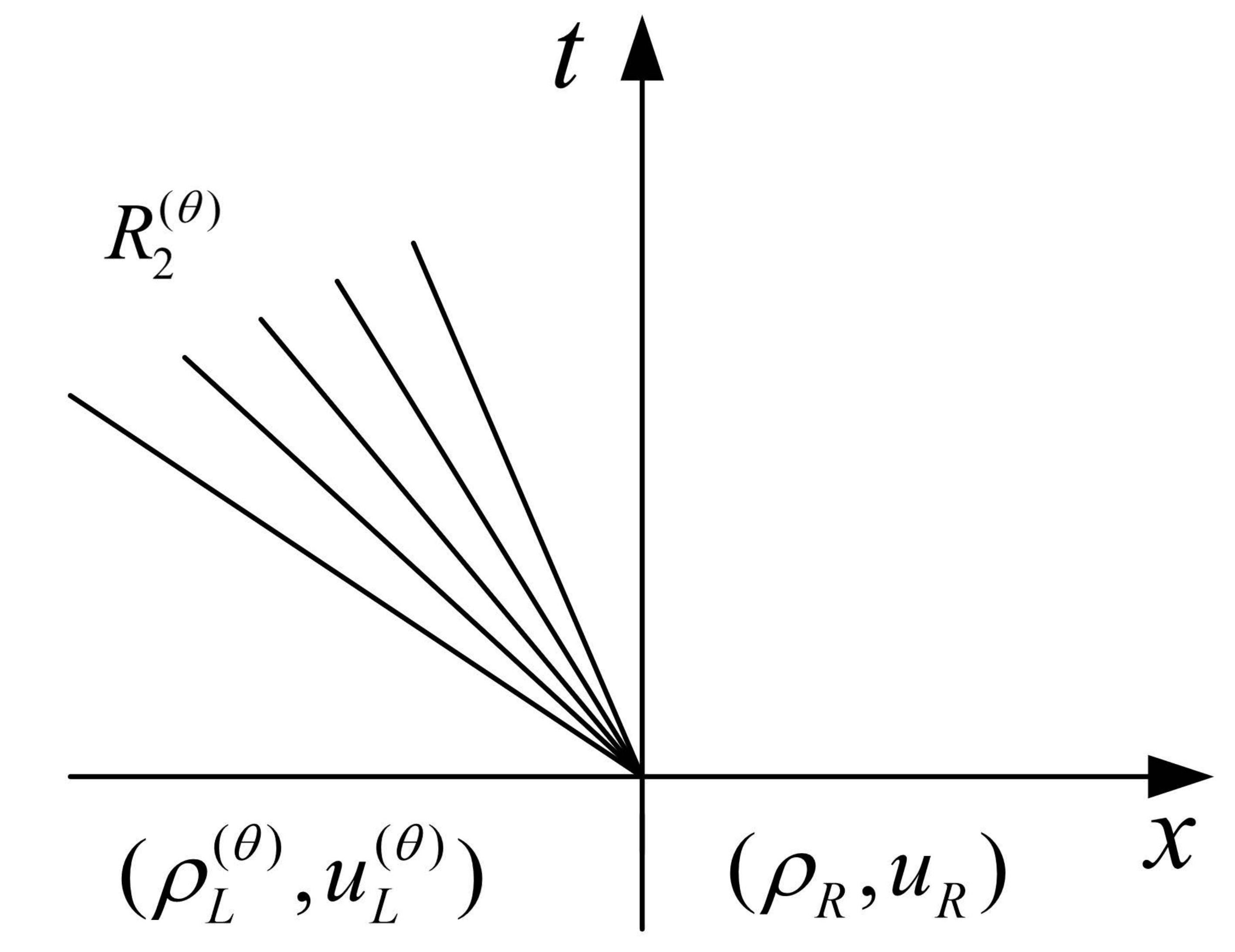}
	\caption{}
	\label{fig:vacuumthetaapp}
\end{figure}

Then the Riemann solutions $(\rho^{(\theta)}_A, u^{(\theta)}_A)$ as shown in Fig. \ref{fig:vacuumthetaapp} is
\begin{equation*}
(\rho^{(\theta)}_A, u^{(\theta)}_A)
=\begin{cases}
(\rho^{(\theta)}_{L}, u^{(\theta)}_L)\qquad\,&\mbox{for $\frac{x}{t}\leq u^{(\theta)}_L+(\rho^{(\theta)}_L)^\theta$},\\[1mm]
((\frac{\theta\frac{x}{t}-\theta u_L+\rho_R^\theta}{\theta+1})^{\frac{1}{\theta}},
\frac{\frac{x}{t}+\theta u_L-\rho_R^\theta}{\theta+1})\qquad\,&\mbox{for $u^{(\theta)}_L+(\rho^{(\theta)}_L)^\theta< \frac{x}{t}< u_R+ \rho_R^\theta$},\\[1mm]
(\rho_{R}, u_R)\qquad\,&\mbox{for $\frac{x}{t}\geq u_R+ \rho_R^\theta$}.
\end{cases}
\end{equation*}
When $\theta\rightarrow 0$, we see that $\rho^{(\theta)}_L\rightarrow 0$, $u^{(\theta)}_L\rightarrow -\infty$,
$u^{(\theta)}_L+\big(\rho^{(\theta)}_L\big)^\theta\rightarrow-\infty$,
and $u_{R}+\rho^{\theta}_{R}\rightarrow u_R+1$. Similarly,
we conclude that  $(\rho^{(\theta)}_A, u^{(\theta)}_A)\rightarrow (\rho^{(0)}, u^{(0)})$.

\medskip
\bigskip
\noindent {\bf Acknowledgments:}
The research of
Gui-Qiang G. Chen was supported in part by
the UK
Engineering and Physical Sciences Research Council Awards
EP/L015811/1, EP/V008854/1, and EP/V051121/1.
The research of Fei-Min Huang was supported by the NSFC Grant No.  12288201.
The research of Tian-Yi Wang was supported in part by the NSFC Grants
11971024 and 12061080.

\end{document}